\documentclass[12pt]{article}
\usepackage{amsmath}
\usepackage{amsthm}
\usepackage{amsfonts}
\usepackage{amssymb}
\usepackage{eucal}
\newcommand{\beq}{\begin{equation}}
\newcommand{\eeq}{\end{equation}}
\numberwithin{equation}{section}

\newtheorem{theorem}{Theorem}[section]
\newtheorem{definition}[theorem]{Definition}
\newtheorem{lemma}[theorem]{Lemma}
\newtheorem{remark}[theorem]{Remark}
\newcommand{\RP}[1]{\mathbb{RP}^{#1}}

\newcommand{\tr}{\mathrm{tr}\,}
\newcommand{\bm}{{\bf m}}

\newcommand{\thmref}[1]{Theorem~\ref{#1}}

\newcommand{\lemref}[1]{Lemma~\ref{#1}}

\newcommand{\R}[1]{\mathbb{R}^{#1}}

\newcommand{\eps}{\varepsilon}
\newcommand{\abs}[1]{\left\vert{#1}\right\vert}

\def\intav#1{\mathchoice
          {\mathop{\vrule width 6pt height 3 pt depth -2.5pt
                  \kern -9pt \intop}\nolimits_{\kern -6pt#1}}%
          {\mathop{\vrule width 5pt height 3 pt depth -2.6pt
                  \kern -6pt \intop}\nolimits_{#1}}%
          {\mathop{\vrule width 5pt height 3 pt depth -2.6pt
                  \kern -6pt \intop}\nolimits_{#1}}%
          {\mathop{\vrule width 5pt height 3 pt depth -2.6pt
                  \kern -6pt \intop}\nolimits_{#1}}}

\def\XXint#1#2#3{{\setbox0=\hbox{$#1{#2#3}{\int}$}
      \vcenter{\hbox{$#2#3$}}\kern-.5\wd0}}

\begin{document}

\title{On Minimizers of a Landau-de Gennes Energy Functional on Planar Domains}
\author{
{\Large Dmitry Golovaty\footnote{Department of Mathematics, Buchtel College of Arts and Sciences, The University of Akron, Akron, OH, 44325, USA. \ Supported in part by the NSF grant DMS-1009849 {\tt dmitry@uakron.edu}}} \and
{\Large Jos\'e Alberto Montero\footnote{Facultad de Matem\'aticas, Pontificia Universidad Cat\'olica de Chile,
Vicu\~na Mackenna 4860, San Joaqu\'in, Santiago, Chile. \ Supported by Fondecyt Grant No. 1100370 {\tt amontero@mat.puc.cl}}}}
\date{\today}
\thispagestyle{empty}
\maketitle
\begin{abstract}
We study tensor-valued minimizers of the Landau-de Gennes energy functional on a simply-connected planar domain $\Omega$ with non-contractible boundary data.  Here the tensorial field represents the second moment of a local orientational distribution of rod-like molecules of a nematic liquid crystal. Under the assumption that the energy depends on a single parameter---a dimensionless elastic constant $\eps>0$---we establish that, as $\eps\to0$, the minimizers converge to a projection-valued map that minimizes the Dirichlet integral away from a single point in $\Omega$.  We also provide a description of the limiting map.
\end{abstract}

\section{Introduction}

In this paper we study minimizers of the Landau-de Gennes (LdG) energy functional in the presence of disclinations.  Under the assumptions that will be discussed later in the introduction, the corresponding variational problem can be described as follows. Let $\Omega \subset \R{2}$ be a smooth, bounded, and simply-connected domain and denote by $F_1$ the set of symmetric $3\times 3$ matrices with trace $1$.  For each $u\in W^{1,2}(\Omega, F_1)$, set
\begin{equation}\label{E_eps_def}
E_\eps(u) = \int_\Omega \left ( \frac{\abs{\nabla u}^2}{2} + \frac{W(u)}{\eps^2}\right ).
\end{equation}
Here $\eps > 0$ is a small parameter and
\begin{equation}
\label{W}
W(u) = {\rm tr}(q(u))=\frac{1}{2}{\rm tr}\left(\left(u-u^2\right)^2\right).
\end{equation}
Observe that $W(u)\geq 0$ for any $u \in F_1$ and $W(u)=0$ if and only if $u \in \cal P,$ where
$$
{\cal P} = \{A\in F_1: A^2=A\}
$$
is the set of rank-one, orthogonal projection matrices. 

Our results concern the minimizers $u_\eps$ of $E_\eps$ among $u \in W^{1,2}(\Omega, F_1)$ that satisfy $u = g$ on $\partial \Omega$ for {\em topologically nontrivial} boundary data $g$ corresponding to non-contractible curves in $\cal P$.  Our first result establishes the existence of a single point $a$ in the interior of $\Omega$ such that the $u_\eps$ converge to a function $u_0\in W^{1,2}_{loc}(\Omega \setminus \ \{a\}, {\cal P})$ as $\eps \to 0$.  More precisely, we prove the following

\begin{theorem}
\label{one_sing}
Let $g : \partial \Omega \to \cal P$ be a non-contractible curve in $\cal P$ and suppose that $u_\eps \in W^{1,2}(\Omega; F_1)$ is a minimizer of $E_\eps$ among functions $u \in W^{1,2}(\Omega; F_1)$ that satisfy the Dirichlet boundary condition $u = g$ on $\partial \Omega.$  First, the minimizers $u_\eps$ take values in the convex envelope of $\cal P$; in particular they are uniformly bounded in $\eps$.  Second, there is a single point $a$ in the interior of $\Omega$ such that the $u_\eps$ converge strongly (along a subsequence) to $u_0 \in W^{1,2}(\Omega\setminus B_R\{a\}; {\cal P})$ in $W^{1,2}(\Omega\setminus B_R\{a\}; F_1)$ as $\eps \to 0$ for any fixed $R >0$.  Finally, for any open set $U \subset \subset \overline{\Omega}\setminus \{a\}$, $u_0$ minimizes $\int_U \abs{\nabla v}^2$ among functions $v \in W^{1,2}_{loc}(\Omega\setminus \{a\}; {\cal P})$ satisfying $v = u_0$ on $\partial U$.
\end{theorem}

To describe the structure of $u_0,$ let $M_a^3(\R{})$ be the set of antisymmetric $3\times 3$ matrices and let $[A;B]=AB-BA$ denote the commutator of matrices $A$ and $B$. It turns out that one can consider a vector field $j(u_0)$ with matrix entries
$$
j(u_0) = \left (\left [ u_0 ; \frac{\partial u_0}{\partial x}\right ], \left [ u_0 ; \frac{\partial u_0}{\partial y}\right ] \right ),
$$
instead of $u_0$ because $u_0$ can always be recovered from $j(u_0)$ (the reason for this reduces to the following standard fact: if $A:[0,T]\to M_a^3(\R{})$, then the solution of the initial value problem
$$
\gamma' = [\gamma; A],\ \ \gamma(0) \in\cal P,
$$
takes values in $\cal P$). In light of this observation, the following theorem gives a rough description of the limiting map $u_0$ described in \thmref{one_sing}.

\begin{theorem}
\label{main_theo}
Let $u_0$ be as in \thmref{one_sing}.  There is a function 
$$\psi_0 \in (W^{1,2}\cap L^\infty)(\Omega; M_a^3(\R{}))$$ 
and a constant anti-symmetric matrix $\Lambda_0$ such that
$$
j(u_0) = \frac{1}{2\pi r_a} (\hat{\theta}_a \Lambda_0)  + \nabla^\perp \psi_0
$$
in $\Omega.$ Here $a \in \Omega$ is as defined in \thmref{one_sing}, $r_a$ and $\hat{\theta}_a$ are the radial variable and the unit vector in an angular direction for polar coordinates centered at $a$ respectively, and we interpret $(\hat{\theta}_a \Lambda_0)$ and $\nabla^\perp \psi_0$ as matrix-valued vector fields according to (\ref{vector_matrix_entries}) and (\ref{def_grad_perp}) respectively.  Further, $\psi_0$ satisfies
$$
\Delta \psi_0 = 2\left [ \frac{\partial \psi_0}{\partial x}; \frac{\partial \psi_0}{\partial y}\right ] + \frac{1}{\pi r_a} \left [ \nabla \psi_0 \cdot \hat{\theta}_a ; \Lambda_0\right ],
$$
in $\Omega$, where we interpret $\nabla \psi_0 \cdot \hat{\theta}_a$ according to (\ref{def_F_dot_e}), subject to boundary conditions
$$
-\nabla \psi_0 \cdot \nu = \left [g; \frac{dg}{d \tau}\right ] - \frac{\hat{\theta}_a\cdot \tau}{2\pi r_a} \Lambda_0,
$$
on $\partial \Omega,$ where $\nu$ and $\tau$ are the outward unit normal and unit tangent vector to $\partial\Omega,$ respectively.  Finally, the function
$
Z_{u_0}(x) := \frac{1}{2\pi r_a}(\Lambda_0 - u_0 \Lambda_0 - \Lambda_0 u_0) \in L^2(\Omega; M_a^3(\R{})).
$
\end{theorem}

Although we cannot prove it yet, we conjecture that the map $\psi_0$ in \thmref{main_theo} is smooth. If in addition $\nabla \psi_0(a)=0,$ our results allow for renormalization of $E_\eps(u_\eps)$ along the lines of \cite{BBH} via an expansion of $E_\eps(u_\eps)$ containing a leading term proportional to $\abs{\ln{\eps}}$ and bounded terms depending only on $a$ and the boundary data $g$.  

Our problem is closely related to and motivated by the studies of equilibrium configurations of nematic liquid crystals---materials composed of rod-like molecules that flow like fluids, yet they retain a degree of molecular orientational order similar to crystalline solids. There are several mathematical frameworks to study the nematics, leading to different, but related variational models that we will discuss next.

The local orientational order can be described by specifying a director---a unit vector in a direction preferred by the molecules at a given point. The director field forms a basis for the Oseen-Frank theory for the uniaxial nematic liquid crystals \cite{virga}. Within this theory, one constructs an energy penalizing for spatial variations of the director, distinguishing between various elastic modes (splay, bend, twist) and taking into account interactions with electromagnetic fields. Although this theory has generally been very successful in predicting equilibrium nematic configurations, it prohibits certain types of topological defects, e.g., disclinations, as the constraint that the director must have a unit length becomes too rigid. A possible remedy was proposed by Ericksen \cite{ericksen} who introduced a scalar parameter intended to describe the quality---the degree---of local molecular orientational order. 

Despite the fact that the Ericksen's theory is capable of handling line defects, it still assumes that a preferred direction is specified by the director, excluding a possibility that the nematic can be biaxial. Here a biaxial state differs from a uniaxial state in that it has no rotational symmetry; instead it possesses reflection symmetries with respect to each of a three orthogonal axes (only two of which need to be specified). Biaxial configurations are conjectured to exist, e.g., at the core of a nematic defect. Further, certain nematic configurations cannot even be orientable, that is, they cannot be described by a continuous director field \cite{Ball_Zarn}. These deficiencies can be circumvented within the Landau-de Gennes theory that we will now briefly review (see also \cite{Ball_Zarn}, \cite{apala_zarnescu_01}, and \cite{Mottram_Newton}).

Suppose that orientations of rod-like molecules in a small neighborhood of a point $x\in\Omega$ can be described in terms of a probability density function $\psi(x,\bm):\Omega\times\mathbb S^2\to\mathbb R^+,$ i.e., the probability that the molecules near $x$ are oriented within a subset $S\in\mathbb S^2$ is given by
\[
p(x,S)=\int_S\psi(x,\bm)\,d\sigma\,.
\] 
Since the head and tail of a nematic molecule are indistinguishable, the function $\psi(x,\cdot)$ is even and the first moment of $\psi(x,\cdot)$ vanishes. Consequently, if one were to seek a macroscopic theory based on moments of $\psi(x,\cdot),$ the simplest approach would be to use the second moment
\begin{equation}
\label{pdf}
u(x)=\int_{\mathbb S^2}{\bm\otimes\bm}\,\psi(x,\bm)\,d\sigma,
\end{equation}
where $({\bf a}\otimes{\bf b})_{ij}=a_ib_j,\ i,j=1,\ldots,3$ is the tensor product of $\bf a$ and $\bf b$. 

The following properties of $u$ immediately follow from \eqref{pdf} and the fact that $\psi(x,\cdot)$ is a probability density function
 \begin{enumerate}
\item $u(x)\in F_1$ and its eigenvalues satisfy $\lambda_i\in[0,1],\ i=1,\ldots,3$.
\item $u(x)=\frac{1}{3}I$ in an {\em isotropic state} when all molecular orientations in a vicinity of $x$ are equally probable, i.e., $\psi(x,\bm)=\frac{1}{4\pi}.$ Here $I$ is the identity matrix.  
\item $u(x)=\bm_0\otimes\bm_0\in\cal P$ in a {perfect uniaxial nematic state} when all molecules near $x$ are parallel to $\pm\bm_0,$ i.e., $\psi(x,\bm)=\frac{1}{2}(\delta(\bm-\bm_0)+\delta(\bm+\bm_0)).$
\end{enumerate}
For other forms of $\psi(x,\cdot),$ the set of eigenvalues of $u(x)$ differs from those in (i) and (ii) and the nematic is in intermediate states of order that can be either uniaxial with a degree of orientation less than $1$ or biaxial. 

In {\em thermotropic nematics}, a phase transition from an isotropic to a nematic state occurs as the temperature is decreased below a certain threshold value. The best way to account for a symmetry change during the transition is to define an appropriate order parameter. Within the LdG phenomenological theory, the role of the order parameter is played by the second order tensor $Q$ equal to $u$ defined in the previous paragraph, translated by a factor of $\frac{1}{3}I$. The theory is based on the hypothesis that equilibrium properties of the system can be found from a non-equilibrium free energy, constructed as an $O(3)$-symmetric expansion in powers of $Q$.

In this paper we formulate our results in terms of the matrix $u$ for the reasons of mathematical simplicity, although they can easily be restated within a standard $Q-$tensor framework by incorporating the appropriate translation. We will further assume that the lowest energy configuration at temperatures below the isotropic-nematic transition is that of a perfect uniaxial nematic $u\in\cal P,$ while the isotropic state $u=\frac{1}{3}I$ minimizes the energy above the transition temperature. Since the LdG free energy must be invariant with respect to rotations, it can only be a function of the invariants of the matrix $u.$ Given these conditions and incorporating the invariants to the least possible powers, we obtain that
\begin{equation}
\label{wisborn}
W_\beta(u)=2I_2^2(u)-\beta I_3(u)\,,
\end{equation}
where the invariants are given by
\[I_2(u)=\frac{1}{2}\left(1-\tr{(u^2)}\right),\ \ I_3(u)={\rm det}(u) = \frac{1}{6}\left(1-3\tr{(u^2)}+2\,\tr{\left(u^3\right)}\right),\]
since the trace of $u$ is equal $1.$ Simple calculations show that a perfect uniaxial state $u\in\cal P$ is a local minimum of $W_\beta$ when $0<\beta<8$ and it is a global minimum of $W_\beta$ when $\beta\leq6$. The isotropic state $u=\frac{1}{3}I$ is a local maximum of $W_\beta$ when $0<\beta\leq4$, it is a local minimum of $W_\beta$ when $4<\beta<6,$ and it is a global minimum of $W_\beta$ when $6\leq\beta\leq8.$ Note that the expression \eqref{wisborn} is equivalent to the standard LdG energy for the traceless tensors, once the condition that the nematic minimum corresponds to a perfect uniaxial state is imposed. Since only two out of the three coefficients in the standard energy can be imposed independently, the additional condition reduces the number of the coefficients to one and $\beta$ above should be temperature-dependent. In this work we will assume that $2<\beta<6,$ i.e., the temperature is below that of the nematic-to-isotropic transition. The lower bound on $\beta$ will be explained later on in the text---it is related to the fact that predictions on the phenomenological, expansions-based LdG theory become non-physical away from the transition temperature (cf. \cite{apala_nonphys}). Unless specified otherwise, for simplicity we will set $\beta=3,$ thus recovering \eqref{W}.

The spatial variations of the order parameter in the LdG theory are controlled by the term quadratic in the gradient of the order parameter. Here we will assume that all elastic constants are equal so that this part of the energy becomes proportional to the Dirichlet integral. Finally, we assume that the remaining (non-dimensional) elastic constant $\epsilon$ is small---e.g., when the diameter of $\Omega$ is large---and that the three-dimensional cylindrical domain $\Omega\times[-L,L]$ occupied by the liquid crystal and the boundary data are such that we can ignore the dependence on the axial spatial variable.

The Dirichlet boundary conditions on $u$ are referred to as the strong anchoring conditions on $\partial\Omega$ in the physics literature: they impose specific preferred orientations on nematic molecules on surfaces bounding the liquid crystal. We are interested in a situation in which the nematic is in a perfect uniaxial state on the boundary and has a winding number $\pm\frac{1}{2}$; in this case the nematic has a disclination in $\Omega\times[-L,L]$ or, equivalently, a point defect/vortex in $\Omega$.

To summarize the discussion above, we consider a variational problem for an energy functional $E_\eps$ given in (\ref{E_eps_def}) that describes a nematic liquid crystal within the context of the Landau-de Gennes theory. The functional is defined over the set of matrix-valued functions; the principal contribution of this work is that {\em we do not impose any constraints on the target set} $F_1$ of $3\times 3$ symmetric, trace-one matrices, beyond what is required by the LdG theory. The variational problem consists of minimizing $E_\eps$ among all $u \in W^{1,2}(\Omega;F_1)$ that are subject to the Dirichlet boundary condition $u = g$.  Here $\Omega \subset \R{2}$ is a bounded, smooth, simply-connected domain and $g:\partial \Omega \to \cal P$ represents a non-contractible curve in the set of rank-one, orthogonal projection matrices.  Our main goal is to understand the behavior of the minimizers of $E_\eps$ in the limit of a vanishing elastic constant $\eps \to 0$.

Whenever possible, our approach follows the roadmap established for Ginzburg-Landau vortices by Bethuel, Brezis and Helein in \cite{BBH}.  As in that work, we find that the minimizers $u_\eps$ of $E_\eps$ have energies that blow up as $\abs{\ln(\eps)}$ when $\eps \to 0.$ The minimizers in \cite{BBH} converge to an $S^1$-valued harmonic map away from a finite set of points in $\Omega$ in the limit of $\eps \to 0.$ The situation is similar here, although the limiting harmonic map is $\cal P$-valued and the singular set consists of a single point. On the other hand, even though the energy-estimates-based techniques from \cite{BBH} can mostly be extended to our case (albeit, nontrivially), the principal difference between this work and \cite{BBH} is that the results in \cite{BBH} that rely on the structure of harmonic maps into $S^1$ are no longer applicable to harmonic maps with values in $\cal P$ (or, equivalently, to $\RP{2}$). Inspired by H\'elein's treatment of the B\"acklund transformation \cite{Helein_moving_frames}, instead of studying the limiting map directly, we choose to describe it in terms of its current vector (\ref{def_current_vector}). This leads us to consider solutions of the CMC equation \cite{riviere} that are not in $W^{1,2}$.  The connection between the CMC equation and the LdG energy seems to have not been made in the literature before.

In a recent work \cite{bauman_phillips_park}, Bauman, Park and Phillips considered a related problem for an energy functional with a more general expression for the elastic energy that is defined over a more narrow admissible class of functions. The mathematical problem in \cite{bauman_phillips_park} describes a thin nematic film with the strong orthogonal anchoring on the surfaces of the film. The anchoring forces one eigenvector of the order parameter matrix inside the film to be perpendicular to the film surface. The limiting map in \cite{bauman_phillips_park} then takes values in $\RP{1}$ making the analysis of \cite{bauman_phillips_park} closer to that of \cite{BBH} than what is possible for our problem. On the other hand, the additional constraint on the admissible space of functions allows for a comparatively better description of the limiting map. 

Note that, when $\Omega \subset \R{3},$ the convergence analysis for $u_\eps$ is quite different from its two-dimensional counterpart. Indeed, although the limiting map from $\R{3}$ into $\cal P$ can also have singularities, the energies $E_\eps(u_\eps)$ of the minimizers $u_\eps$ are uniformly bounded as $\eps \to 0$.  The interested reader can find a thorough review of recent work on this problem in \cite{CULect}.

After this work was submitted for publication, we had learned that results similar to our Theorem 1 have been simultaneously obtained by Canevari \cite{canevari}.  However, the methods in \cite{canevari} are significantly different from ours in that the author intentionally avoids using the matrix algebra of the problem, whereas we use it extensively.

The manuscript is organized as follows: in the next section we set our notation and collect some well known-facts needed for subsequent developments.  In Section 3 we prove \thmref{one_sing}.  In the last section we prove \thmref{main_theo}.

\section{Notation}

In this section we set our notation.  We will denote by $M^3(\R{})$ the set of $3\times 3$ matrices with real entries and by $M^3_a(\R{})$, $M^3_s(\R{}),$ and $O(3)$ the sets of anti-symmetric, symmetric, and orthogonal matrices, respectively.  For any pair $A,B \in M^3(\R{}),$ we set
$$
\langle A,B \rangle = {\rm tr}(A^TB),\,\,\, \abs{A}^2 = \langle A,A \rangle \,\,\,\mbox{and}\,\,\, [A;B]=AB-BA.
$$
$I_n$ will denote the $n\times n$ identity matrix, whereas we will write $I_{3\times 3}$ for the identity map from $M^3(\R{})$ to itself.

For any $a \in \R{2}$, the standard polar coordinates centered at $a$ will be denoted by $r_a$, $\theta_a,$ with $\hat{r}_a$, $\hat{\theta}_a$ being the corresponding unit vectors (we will drop the subscript whenever there is no ambiguity).

The set of rank-one orthogonal projections in $\R{3}$ will be denoted by $\cal P$, that is,
$$
{\cal P} = \{ A \in M^3(\R{}): A^T = A^2 = A, \,\,\, {\rm tr}(A) = 1\}.
$$

It is well known that $\cal P$ is diffeomorphic to the real projective space.  We define
$$
L_0 = \inf \{l(\gamma): \,\,\, \gamma \,\,\, \mbox{a closed, non-contractible curve in}\,\,\, \cal P\},
$$
where $l(\gamma)$ denotes the length of the curve $\gamma$.  With the usual ($2$ to $1$) covering map from $\mathbb S^2$ to $\cal P,$ we can associate every closed geodesic in $\cal P$ with a great circle in $\mathbb S^2$, thus $L_0 > 0$.  Further, if $\gamma_1, \gamma_2$ are two closed geodesics in $\cal P$, there is an orthogonal {\em constant} matrix $R\in O(3)$ such that $\gamma_2 = R\gamma_1 R^T$.  

Let now
\begin{equation}\label{canonical_flat}
\gamma_0(t) = \frac{1}{2} \left  ( I_3 + \left ( \begin{array}{ccc} \cos(t) & \sin(t) & 0\\ \sin(t) & -\cos(t) & 0 \\ 0 & 0 & -1 \end{array} \right ) \right ),
\end{equation}
represent a closed, non-contractible geodesic in $\cal P$.  A direct computation shows that
\begin{equation}\label{A_0_constant}
A_0(t) = \frac{1}{2\pi}\left [ \gamma_0(t); \frac{d\gamma_0}{dt}(t)\right ]
\end{equation}
is a constant. Since any other closed geodesic in $\cal P$ can be written as $\gamma_1 = R\gamma_0 R^T$, where $R\in O(3)$ is a constant orthogonal matrix,
$$
A(t) = \frac{1}{2\pi}\left [ \gamma(t) ; \frac{d\gamma}{dt}(t)\right ] = \frac{1}{2\pi}R\left [ \gamma_0(t) ; \frac{d\gamma_0}{dt}(t)\right ]R^T
$$
is constant for any closed geodesic $\gamma$ in $\cal P$, and
$$
\abs{A}^2 = \abs{A_0}^2.
$$
Consider now any matrix $A\in M_a^3(\R{})$ that can be written as $A=RA_0 R^T$, where $A_0$ is given by (\ref{A_0_constant}) and $R\in O(3)$.  Next, solve the system of ODEs
$$
\gamma' = [A; \gamma]
$$
with the initial condition $\gamma(0) = R\gamma_0(0)R^T$.  By a uniqueness theorem for this ODE, the solution $\gamma = R\gamma_0 R^T$ is a closed geodesic.  A direct computation shows that $A=[\gamma;\gamma']$.

The previous discussion demonstrates that there is a $1-1$ correspondence between closed geodesics in $\cal P$ and antisymmetric matrices of the form $A=RA_0R^T$ with $R\in O(3)$.  We will call such an $A\in M_a^3(\R{})$ an antisymmetric representative of a geodesic.

Set now
$$
F_\lambda = \{A \in M_s^3(\R{}): {\rm tr}(A) = \lambda\}.
$$
We will denote by
$$
\Sigma \,\,\,\,\mbox{and} \,\,\,\, \Pi
$$
the closed convex envelope of $\cal P$ in $F_1$, and the projection from $F_1$ onto $\Sigma$, respectively.

We shall make use of the following
\begin{definition}\label{def_min_rotation} Let $A,B \in \cal P$, $A \neq B$, be any two matrices such that $\langle A,B\rangle \neq 0$.  The {\em minimal rotation} $R(A,B)$ mapping $A$ to $B$ is the unique matrix $R \in O(3)$ such that
$$
B=RAR^T
$$
and
$$
RCR^T = C
$$
for the unique matrix $C\in \cal P$ with $CA=AC=0$ and $CB=BC=0$.  If $A = B$, we define $R(A,B) = I_3$.
\end{definition}
\begin{remark}
We emphasize that
$$
B = R(A,B) A R^T(A,B)
$$
for any two $A, B \in \cal P$ such that the angle between their images is not $\frac{\pi}{2}$.
\end{remark}
We will make use of the fact that $R(A,B)$ depends smoothly on $A,B \in \cal P$, at least when $A$ and $B$ are close to each other.  This can be seen from the next
\begin{lemma}\label{expr_min_rot}  Let $A, B \in \cal P$ be such that $\langle A,B\rangle \neq 0$, then
\begin{align}
R(A,B) &= {\rm exp}\left ( \frac{1}{2}\ln \left (  I_3  + \frac{2}{\langle A, B\rangle}[A;B]^2 + 2[A;B]\right ) \right ),
\end{align}
where $\ln$ denotes a local inverse of the exponential map ${\rm exp}:M_a^3(\R{})\to O(3)$ near $I_3$.
\end{lemma}

\begin{proof}

This expression is easy to establish if we take $A=\gamma_0(\alpha)$ and $B=\gamma_0(\beta)$, where $\gamma_0$ is given in (\ref{canonical_flat}).  However, any two $A, B \in \cal P$ can be written in this form in some coordinate system.  This proves the lemma.

\qed
\end{proof}

Let now $\Omega \subset \R{2}$.  We will often deal with {\em matrix-valued functions} $u:\Omega \to M^3(\R{})$ and {\em matrix-valued vector fields} $F : \Omega \to (M^3(\R{}))^2$, $F = (F_1,F_2)$.  For a matrix-valued function $u$, the gradient and its perpendicular are given by the matrix-valued vector fields
\begin{equation}\label{def_grad_perp}
\nabla u = \left ( \frac{\partial u}{\partial x}, \frac{\partial u}{\partial y} \right ), \,\,\, \nabla^\perp u = \left ( \frac{\partial u}{\partial y}, -\frac{\partial u}{\partial x} \right ),
\end{equation}
respectively.  For matrix-valued vector fields, the divergence and curl
$$
\nabla \cdot F = \frac{\partial F_1}{\partial x} + \frac{\partial F_2}{\partial y}, \,\,\, \nabla^\perp \cdot F = \frac{\partial F_2}{\partial x} - \frac{\partial F_1}{\partial y}
$$
are matrix-valued functions.   When $z :\Omega \to \R{2}$ and $A:\Omega \to M^3(\R{})$, the matrix-valued vector field $zA$ has the entries
\begin{equation}\label{vector_matrix_entries}
zA = (z_1A, z_2A).
\end{equation}
On the other hand, if $F$ is a matrix-valued vector field and $e=(e_1, e_2)\in \R{2}$, we set
\begin{equation}\label{def_F_dot_e}
F\cdot e = e_1F_1+e_2F_2,
\end{equation}
which is a matrix-valued function.  We emphasize the difference between $F\cdot e$ and $zA$ defined in (\ref{vector_matrix_entries}). In what follows, unless there is an ambiguity, we will refer to matrix-valued functions and matrix-valued vector fields simply as functions and vector fields, respectively.

Given a function $u\in W^{1,2}(\Omega, M^3(\R{}))$ its {\em current vector field} is
\begin{equation}\label{def_current_vector}
j(u) = (j_1,j_2) = \left ( \left [ u ; \frac{\partial u}{\partial x}\right ], \left [ u ; \frac{\partial u}{\partial y}\right ] \right ),
\end{equation}
which can be written informally as
$$
j(u) = [u; \nabla u].
$$
Notice that
$$
\nabla \cdot j(u) = [u; \Delta u] \,\,\,\mbox{and}\,\,\, \nabla^\perp \cdot j(u) = 2\left [\frac{\partial u}{\partial x}; \frac{\partial u}{\partial y}\right ].
$$

Whenever $u :\Omega \to {\cal P}$, differentiating the identity $u^2=u$ and performing some simple computations, we obtain 
$$
\abs{j(u)}^2 = \abs{\left [ u ; \frac{\partial u}{\partial x}\right ]}^2 + \abs{\left [ u ; \frac{\partial u}{\partial y}\right ]}^2 = \abs{\nabla u}^2.
$$
Also from $u^2 = u$ we have
$$
\left [ \frac{\partial u}{\partial x}; \frac{\partial u}{\partial y}\right ] = -\left [ \left [ u ; \frac{\partial u}{\partial x}\right ] ; \left [ u ; \frac{\partial u}{\partial y}\right ] \right ].
$$
Hence, for $u \in W^{1,2}(\Omega , {\cal P})$ we have
$$
\nabla^\perp \cdot j(u) +2[j_1;j_2]=0.
$$

We are now ready to proceed with the proofs of the main results of the paper.

\section{Proof of \thmref{one_sing}}

The proof will be split into a series of lemmas.  Throughout the remainder of the paper we will fix a smooth open set $\Omega_L$ such that $\Omega \subset \subset \Omega_L$ and there is an extension $u_g$ of the function $u_\eps$ to $\Omega_L$ that depends only on the boundary data $g$ and
$$
\int_{\Omega_L \setminus \Omega} \abs{\nabla u_g}^2
$$
is finite (and, obviously, independent of $\eps$).

We start by proving the following
\begin{lemma}
\label{proj_lemma}
For any $u \in W^{1,2}(\Omega; F_1)$ and $v = (\Pi \circ u)$, we have
$$
E_\eps(v) \leq E_\eps(u).
$$
\end{lemma}
\begin{proof}
Recall that $\Pi$ is the projection onto $\Sigma$, the convex envelope of $\cal P$.  It is well known that $\Pi$ is a Lipschitz function with a Lipschitz constant $L = 1$, hence
$$
\abs{\nabla v} \leq \abs{\nabla u}.
$$
We need to check then that
$$
W(v) \leq W(u).
$$
To this end, let
$$
S = \{x=(x_1,x_2,x_3) \in \R{3}: x_j \geq 0,\,\,\, \sum_{j=1}^3 x_j = 1\},
$$
be a standard simplex in $\R{3}$ and denote by $\mu$ the projection onto $S$ in $\R{3}$.  Let now $u \in F_1$.  Since $u$ is symmetric, there are three projections $P_j \in \cal P$, and real numbers $\lambda_j \in \R{}$, $j=1,2,3$, such that
$$
u = \sum_{j=1}^3 \lambda_jP_j, \,\,\,\, \sum_{j=1}^3 \lambda_j = 1,
$$
and
$$
P_jP_k = P_kP_j = \delta_{j,k}P_k.
$$
Here $\delta_{j,k}$ denotes the Kronecker symbol and the eigenvalues of $u$ are labeled in the decreasing order $\lambda_1 \geq \lambda_2 \geq \lambda_3$.  Note also, that $\sum_{j=1}^3 P_j = I_3$.

We need to prove that $W(\Pi(u)) \leq W(u)$.  We can assume that $\lambda_3 < 0$; otherwise, $u \in \Sigma$ and $\Pi(u) = u$ and there is nothing to prove.

Our first claim is the following:  if $(\mu_1,\mu_2,\mu_3)=\mu(\lambda_1,\lambda_2,\lambda_3)$ denotes the projection of the vector $(\lambda_1,\lambda_2,\lambda_3)$ onto the simplex $S$, then
$$
v = \Pi(u) = \sum_{j=1}^3 \mu_jP_j.
$$
To prove this, first let $Q \in \cal P$ be any rank-one orthogonal projection.  Then we have
\begin{align}
\langle u-v; Q-v\rangle = \sum_{j=1}^3 (\lambda_j - \mu_j)(\langle P_j; Q \rangle - \mu_j) \leq 0 \label{char_proj_one}
\end{align}
because $(\mu_1,\mu_2,\mu_3)=\mu(\lambda_1,\lambda_2,\lambda_3)$ is the projection of $(\lambda_1,\lambda_2,\lambda_3)$ onto $S$ and the vector $(\langle Q; P_1 \rangle, \langle Q; P_2 \rangle, \langle Q; P_3 \rangle) \in S$.  Indeed,
$$
\sum_{j=1}^3 \langle Q; P_j \rangle = \langle Q ; I_3 \rangle = {\rm tr}(Q) = 1,
$$
and $\langle Q; P \rangle \geq 0$ for any $P \in \cal P$.  We observe now that a general $A \in \Sigma$ can be written as
$$
A = \sum_{j=1}^3 \alpha_j Q_j
$$
for some projections $Q_j \in \cal P$ and scalars $\alpha_j \in \R{}$, $j = 1,2,3$ such that
$$
\alpha_j \geq 0, \,\,\, \sum_{j=1}^3 \alpha_j = 1,\,\,\,\, Q_jQ_k=Q_kQ_j = \delta_{k,j}Q_j.
$$
Using this expression, we conclude through (\ref{char_proj_one}) that
\begin{align*}
\langle u-v; A-v\rangle &= \sum_{j=1}^3 (\lambda_j - \mu_j)(\langle P_j; A \rangle - \mu_j) \\
&=\sum_{i=1}^3 \alpha_i \left (  \sum_{j=1}^3 (\lambda_j - \mu_j)(\langle P_j; Q_i \rangle - \mu_j)  \right ) \leq 0
\end{align*}
for all $A \in \Sigma$.  This characterizes the fact that $v = \Pi(u)$.

Hence we need to find $(\mu_1,\mu_2,\mu_3)=\mu(\lambda_1,\lambda_2,\lambda_3)$---the projection of the vector $(\lambda_1,\lambda_2,\lambda_3)$ on the simplex $S$ when
$$
\sum_{j=1}^3 \lambda_j = 1, \,\,\,\, \lambda_3 < 0.
$$
Recall also that the eigenvalues were labeled in the decreasing order $\lambda_1 \geq \lambda_2 \geq \lambda_3$.  We consider the following two cases: $\lambda_3<0$ and either
$$
\lambda_2 + \frac{\lambda_3}{2}\geq 0 \,\,\,\, \mbox{or} \,\,\,\, \lambda_2 + \frac{\lambda_3}{2} < 0.
$$
Case 1:  $\lambda_3 < 0$ and $\lambda_2 + \frac{\lambda_3}{2}\geq 0$.  In this case
$$
\mu_1 = \lambda_1 + \frac{\lambda_3}{2}, \,\,\,\, \mu_2 = \lambda_2 + \frac{\lambda_3}{2}\geq 0 \,\,\,\, \mbox{and}\,\,\,\, \mu_3=0.
$$
Denoting $\lambda = (\lambda_1,\lambda_2,\lambda_3)$ and $\mu = (\mu_1,\mu_2,\mu_3)$, we need to check that
$$
\langle \lambda - \mu; z - \mu \rangle \leq 0,
$$
for any $z \in S.$ We demonstrate this as follows:
\begin{align*}
\langle \lambda - \mu; z - \mu \rangle &= (\lambda_1-\mu_1)(z_1-\mu_1)+(\lambda_2-\mu_2)(z_2-\mu_2)+\lambda_3z_3 \\
&= -\frac{\lambda_3}{2}\left ( z_1 - \lambda_1 -\frac{\lambda_3}{2}\right ) -\frac{\lambda_3}{2}\left ( z_2 - \lambda_2 -\frac{\lambda_3}{2}\right ) + \lambda_3z_3\\
&= -\frac{\lambda_3}{2}(z_1+z_2) + \frac{\lambda_3}{2} + \lambda_3z_3 \\
&= -\frac{\lambda_3}{2}(1-z_3) + \frac{\lambda_3}{2} +\lambda_3z_3 = \frac{3\lambda_3z_3}{2} \leq 0.
\end{align*}

\noindent Case 2: $\lambda_2 + \frac{\lambda_3}{2} < 0$.  In this case $\mu_1 = 1$, $\mu_2 = \mu_3 = 0$.  Again, we need to check that,  for any $z \in S$ we have
$$
\langle \lambda - \mu; z - \mu \rangle \leq 0.
$$
Indeed, 
\begin{align*}
\langle \lambda - \mu; z - \mu \rangle &= (\lambda_1-1)(z_1-1)+\lambda_2z_2+\lambda_3z_3 \\
&= \left ( \lambda_1 +\frac{\lambda_3}{2} - 1\right )(z_1-1) + \left ( \lambda_2 +\frac{\lambda_3}{2}\right )z_2+\lambda_3z_3 \\
&- \frac{\lambda_3}{2}(z_1+z_2-1)\\
&= \left ( \lambda_1 +\frac{\lambda_3}{2} - 1\right )(z_1-1) + \left ( \lambda_2 +\frac{\lambda_3}{2}\right )z_2+\lambda_3z_3 \\
&+ \frac{\lambda_3}{2}z_3.
\end{align*}
We recall now that $\lambda_2+\frac{\lambda_3}{2}<0$.  Since $\lambda_1+\frac{\lambda_3}{2} + \lambda_2+\frac{\lambda_3}{2} = 1$, then $\lambda_1+\frac{\lambda_3}{2} > 1$ and we conclude that $\langle \lambda - \mu; z - \mu \rangle\leq 0$.

Finally we need to verify that $W(u) \geq W(\Pi(u))$ when $\lambda_3<0$.  Notice that in Case 2 above, we have $\Pi(u)=P_1$.  We then have
$$
W(\Pi(u))=0 \leq W(u).
$$
We consider now Case 1 above.  Recall that here we assumed that $\lambda_3 \leq \lambda_2 \leq \lambda_1$, 
$$
\lambda_3 < 0\,\,\,\,\mbox{and}\,\,\,\, \lambda_2 + \frac{\lambda_3}{2} \geq 0.
$$
Recall also that $\lambda_1+\lambda_2+\lambda_3 = 1$.  We will use the notation
$$
s = \frac{\lambda_1 + \lambda_2}{2}, \,\,\,\, t = \frac{\lambda_1-\lambda_2}{2},
$$
and observe that
$$
s = \frac{\lambda_1 + \lambda_2}{2} = \frac{1 - \lambda_3}{2} > \frac{1}{2}.
$$
Using this notation we have that
$$
\lambda_1 + \frac{\lambda_3}{2} = \frac{1 + \lambda_1-\lambda_2}{2} = \frac{1}{2}+t
$$
and
$$
\lambda_2 + \frac{\lambda_3}{2} = \frac{1 + \lambda_2-\lambda_1}{2} = \frac{1}{2}-t.
$$
Since we also have
$$
\lambda_1 = s+t, \,\,\, \lambda_2 = s-t,
$$
direct computations show that
$$
W(u) = \frac{1}{2}((s+t)^2(s+t-1)^2+(s-t)^2(s-t-1)^2 + (2s)^2(2s-1)^2).
$$
On the other hand we have
$$
W(\Pi(u)) = \left ( \frac{1}{2} + t \right )^2\left ( \frac{1}{2} - t \right )^2.
$$
Denote
$$
\psi(s,t) = \frac{1}{2}((s+t)^2(s+t-1)^2+(s-t)^2(s-t-1)^2 + (2s)^2(2s-1)^2),
$$
and observe that
$$
\psi\left ( \frac{1}{2}, t \right ) = W(\Pi(u)).
$$
To show that $W(\Pi(u)) \leq W(u)$ it suffices to show that
$$
\frac{\partial \psi}{\partial s}(s,t) \geq 0
$$
for all $s \geq 1/2$ and all $t \geq 0$.  To this end, notice first that
$$
\psi(s,t)=q(s+t)+q(s-t) + q(1-2s) = q(s+t)+q(s-t) + q(2s),
$$
where $q(t) = \frac{t^2(1-t)^2}{2}$.  Obviously
$$
\frac{\partial \psi}{\partial s}(s,t) = q'(s+t)+q'(s-t)+2q'(2s).
$$
After some algebra we arrive at
$$
\frac{\partial \psi}{\partial s}(s,t) = 6(2s-1)(s(3s-1)+t^2),
$$
which is non-negative for $s \geq 1/2$ and all $t \geq 0$.  This shows that $W(u) \geq W(\Pi(u))$, and completes the proof of the Lemma.
\qed
\end{proof}
\begin{remark}
The conclusion of Lemma \ref{proj_lemma} is valid for the more general form of the potential $W_\beta$ given by \eqref{wisborn} as long as $\beta\geq 2$. We conjecture that the conclusion is false when $0<\beta<2$ because, in this case, a geodesic connecting two nematic minima on the surface of $W_\beta$ expressed as a function of two independent eigenvalues of $u$ partially lies outside of $S$. Since the trace of $u$ is equal $1$, at least one eigenvalue of $u$ is negative outside of $S$---this violates the condition that the eigenvalues of $u$ must be between $0$ and $1$ within the framework of the LdG theory. We conclude that our approach is valid for the range of the parameters corresponding to the physically relevant case. The fact that the LdG theory fails in a deep nematic regime has been previously discussed in \cite{apala_nonphys}. The non-physicality is due to the fact that the LdG free energy is constructed as an expansion of the non-equilibrium free energy in terms of the order parameter near the temperature of the isotropic-to-nematic transition; the expansion no longer has to approximate the original energy away from the transition temperature.
\end{remark}

\begin{remark}
We will use Lemma \ref{proj_lemma} to establish that minimizers of $E_\eps$ take values in the convex hull of $\cal P$. An alternative maximum-principle-type argument showing that {\em critical points} of $E_\eps$ with boundary data in $\cal P$ have values in $S$ is given in the Appendix.
\end{remark}

Next we collect for future reference some well-known facts regarding $Q(u)$---the nearest point projection of $u \in \Sigma$ onto $\cal P$.  We start by choosing $\delta > 0$ such that, if ${\rm dist}(u;{\cal P}) < \delta$, then $Q(u)$ is well-defined and smooth in $u$.  Next, recall the classical expressions
$$
\lambda_1(u) = \sup \{e \cdot (ue) : e \in \R{3}, \,\,\, \abs{e}=1\},
$$
and
$$
\lambda_3(u) = \inf \{e \cdot (ue) : e \in \R{3}, \,\,\, \abs{e}=1\}.
$$
Also, given $A \in M_s^3(\R{})$, its Moore-Penrose inverse will be denoted by $A^\dagger$. Here $A^\dagger$ is the symmetric matrix that has the same kernel as $A$ and is the inverse of $A$ in the subspace of $\R{3}$ where $A$ is non-singular.

\begin{lemma}
The functions $\lambda_1$ and $\lambda_3$ are convex and concave, respectively.  Furthermore, whenever $u \in \Sigma$ is such that ${\rm dist}(u;{\cal P}) < \delta$, we have
$$
(\nabla_u \lambda_1)(u) = Q(u).
$$
Finally
$$
(D_u Q)(u)(A) = (D_u^2 \lambda_1)(u)(A) = -(u-\lambda_1(u)I_3)^\dagger A v-v A (u-\lambda_1(u)I_3)^\dagger .
$$
Here we use the notation $v=Q(u)$ and $(D_u Q)(u)(A)$ for the Jacobian matrix of $Q(u)$ at $u$ acting on $A$.
\end{lemma}
\begin{remark}
For $v \in \cal P$, the expression $(D_vQ)(v)$ is the orthogonal projection from $M_s^3(\R{})$ onto $T_v \cal P$, the tangent plane to $\cal P$ at $v$.
\end{remark}

\begin{proof}
The fact that $\lambda_1$, $\lambda_3$ are convex and concave, respectively, can be obtained via a standard argument.  Furthermore, it is well-known that
$$
\nabla_u \lambda_1(u) = Q(u).
$$
To obtain the last assertion of the lemma, let $v = Q(u)$ and note that
$$
(u-\lambda_1(u)I_3)v = v(u-\lambda_1(u)I_3) = 0.
$$
Denote by $e_{i,j}:=e_i\otimes e_j$ the matrix with $1$ at the $(i,j)$ and zeros everywhere else, and differentiate the left hand side of the equation above with respect to $u_{i,j}$ to obtain
$$
0 = \left (e_{i,j} - \frac{\partial \lambda_1}{\partial u_{i,j}}(u)I_3 \right )v + (u-\lambda_1(u)I_3)\frac{\partial v}{\partial u_{i,j}}.
$$
From here
$$
(u-\lambda_1(u)I_3)\frac{\partial v}{\partial u_{i,j}} = (u-\lambda_1(u)I_3)(I_3-v)\frac{\partial v}{\partial u_{i,j}} = -\left (e_{i,j} - \frac{\partial \lambda_1}{\partial u_{i,j}}(u)I_3 \right )v,
$$
then
$$
(I_3-v)\frac{\partial v}{\partial u_{i,j}} = -(u-\lambda_1(u)I_3)^\dagger\left (e_{i,j} - \frac{\partial \lambda_1}{\partial u_{i,j}}(u)I_3 \right )v = -(u-\lambda_1(u)I_3)^\dagger e_{i,j} v.
$$
Taking transpose we obtain
$$
\frac{\partial v}{\partial u_{i,j}}(I_3-v) = -v e_{j,i} (u-\lambda_1(u)I_3)^\dagger .
$$
Adding these last two equations we obtain
$$
2\frac{\partial v}{\partial u_{i,j}} - v\frac{\partial v}{\partial u_{i,j}} - \frac{\partial v}{\partial u_{i,j}}v = -(u-\lambda_1(u)I_3)^\dagger e_{i,j} v-v e_{j,i} (u-\lambda_1(u)I_3)^\dagger .
$$
We finally recall that $v = Q(u) \in \cal P$,  hence $v = v^2$. Differentiating this expression, we obtain
$$
v\frac{\partial v}{\partial u_{i,j}} + \frac{\partial v}{\partial u_{i,j}}v = \frac{\partial v}{\partial u_{i,j}}.
$$
All this yields
$$
\frac{\partial v}{\partial u_{i,j}} = -(u-\lambda_1(u)I_3)^\dagger e_{i,j} v-v e_{j,i} (u-\lambda_1(u)I_3)^\dagger .
$$
Taking now $A = (a_{i,j}) \in M_s^3(\R{})$, we multiply the above equation by $a_{i,j}$ and add in $i,j$ to obtain
$$
(D_uQ)(u)(A) = -(u-\lambda_1(u)I_3)^\dagger A v-v A (u-\lambda_1(u)I_3)^\dagger,
$$
which is the last conclusion of the lemma.
\qed
\end{proof}

\begin{lemma}\label{Jerrard_form}
There is a distance $r > 0$, a constant $C > 0$, and an integer $n \geq 2$ such that, for any $\partial B_s(a) \subset\subset \Omega_L$ and any $u \in W^{1,2}(\partial B_s(a); \Sigma)$ with ${\rm dist}(u;{\cal P}) < r$ on $\partial B_s(a)$, we have
$$
\int_{\partial B_s(a)} e_\eps(u) \geq \int_{\partial B_s(a)} \left ( \frac{\rho^n \abs{\nabla_\tau Q(u)}^2}{2} + \frac{\abs{\nabla_\tau \rho}^2}{C} + \frac{1}{C\eps^2}\abs{1-\rho}^2\right ).
$$
Here $\rho = \abs{u}$, and $\nabla_\tau$ denotes the tangential derivative on $\partial B_r(a)$.
\end{lemma}
\begin{proof}
To prove this we write $v = Q(u)$, and notice that, if ${\rm dist}(u;{\cal P})$ is small enough, then
$$
{\rm dist}^2(u;{\cal P}) = \abs{u-v}^2 = (1-\lambda_1)^2 + \lambda_2^2 + \lambda_3^2.
$$
Recall that in this lemma we have $u \in \Sigma$, so $1\geq \lambda_1\geq \lambda_2 \geq\lambda_3 \geq 0$.  In particular, if ${\rm dist}(u;{\cal P}) < r$, we have
$$
\abs{1-\lambda_1} = 1-\lambda_1 \leq {\rm dist}(u;{\cal P})< r,
$$
so $\lambda_1 > 1-r$.  We also have $0 \leq \lambda_2 < r$.  This shows that 
\begin{equation}\label{eigenvals_different}
\lambda_1 - \lambda_3 \geq \lambda_1 - \lambda_2 > 1-2r.
\end{equation}
Recall next that
$$
\frac{\partial v}{\partial x_j} = (D_uQ)(u)\left ( \frac{\partial u}{\partial x_j}\right ).
$$
By the previous lemma we have
$$
\frac{\partial v}{\partial x_j} = v \frac{\partial u}{\partial x_j}(\lambda_1(u)I_3-u)^\dagger + (\lambda_1(u)I_3-u)^\dagger \frac{\partial u}{\partial x_j}v.
$$
For $v\in \cal P$ and $A \in M_s^3(\R{})$ we write
$$
T_v(A) = vA(I_3-v) + (I_3-v)Av,
$$
the projection of $A$ onto $T_v \cal P$.  Then, denoting $\nabla_e f = e\cdot \nabla f$ where $e \in \R{2}$ is a unit vector, we observe that
\begin{equation}\label{split_nabla_tan_perp}
\abs{\nabla_e u}^2 = \abs{T_v(\nabla_e u)}^2 + \abs{(I_{3\times 3}-T_v)(\nabla_e u)}^2.
\end{equation}
We recall that $I_{3\times 3}$ denotes the identity in $M^3(\R{})$.  This last equality holds because $T_v$ is an orthogonal projection in $M^3(\R{})$.  Now we have the following
\begin{align*}
\frac{\partial v}{\partial x_j} &= v \frac{\partial u}{\partial x_j}(\lambda_1(u)I_3-u)^\dagger + (\lambda_1(u)I_3-u)^\dagger \frac{\partial u}{\partial x_j}v \\
&= T_v\left ( \frac{\partial u}{\partial x_j}\right ) \\
&+ v\frac{\partial u}{\partial x_j}((\lambda_1I_3-u)^\dagger - (I_3-v))\\
&+((\lambda_1I_3-u)^\dagger - (I_3-v))\frac{\partial u}{\partial x_j}v.
\end{align*}
This shows that
\begin{align}
\abs{\nabla_e v}^2 = &\abs{T_v(\nabla_e u)}^2 \nonumber \\
&+ \abs{ v(\nabla_e u)((\lambda_1I_3-u)^\dagger - (I_3-v))
+((\lambda_1I_3-u)^\dagger - (I_3-v))(\nabla_e u)v}^2 \nonumber \\
&+ 2\langle T_v(\nabla_e u); v(\nabla_e u)((\lambda_1I_3-u)^\dagger - (I_3-v))
\rangle. \label{diff_moore_pen_I_min_v}
\end{align}
Let us recall now that
$$
u = \lambda_1 v + \lambda_2v_2 + \lambda_3v_3.
$$
for some rank-one projections $v_2, v_3 \in \cal P$ (with $v_iv_j=v_jv_i=\delta_{i,j}v_j$, $j=1$, ..., $3$, $v_1=v$) since $u \in \Sigma$.
Thus we have
$$
(\lambda_1I_3 - u)^\dagger = \frac{1}{\lambda_1-\lambda_2}v_2+\frac{1}{\lambda_1-\lambda_3}v_3,
$$
and also
$$
(\lambda_1I_3 - u)^\dagger - (I_3-v) = \left ( \frac{1}{\lambda_1-\lambda_2} -1 \right )v_2+\left (\frac{1}{\lambda_1-\lambda_3}-1 \right )v_3.
$$
Notice that, for $r>0$ small enough and ${\rm dist}(u;{\cal P}) < r$, this expression makes sense because of (\ref{eigenvals_different}).

Using the fact that the $\lambda_j$ are decreasing in $j$ and that $\lambda_2 \leq 1-\lambda_1$ we see that
$$
\lambda_1-\lambda_2 \leq \lambda_1-\lambda_3,
$$
and then
$$
\abs{(\lambda_1I_3 - u)^\dagger - (I_3-v)} \leq 2\frac{1-\lambda_1+\lambda_2}{\lambda_1-\lambda_2} \leq \frac{C (1-\lambda_1)}{1-2r}.
$$
Here the constant $C > 0$ is independent of $u$ and $r>0$.

Using this expression we can now go back to (\ref{diff_moore_pen_I_min_v}) to obtain
$$
\abs{\nabla_e v}^2 \leq \abs{T_v(\nabla_e u)}^2 + \frac{C(1-\lambda_1)}{1-2r}\abs{\nabla_e u}^2.
$$
By choosing, for example $0 < r < 1/4$, we get
\begin{equation}\label{est_v}
\abs{\nabla_e v}^2 \leq \abs{T_v(\nabla_e u)}^2 + C(1-\lambda_1)\abs{\nabla_e u}^2,
\end{equation}
where $C > 0$ is independent of $r \in ]0,1/4]$.

Next, we observe that 
$$
\frac{\partial \abs{u}^2}{\partial x_j} = 2\abs{u} \frac{\partial \abs{u}}{\partial x_j} = 2\left\langle u; \frac{\partial u}{\partial x_j} \right\rangle.
$$
In other words, 
$$
\frac{\partial \abs{u}}{\partial x_j} = \left\langle \frac{u}{\abs{u}}; \frac{\partial u}{\partial x_j} \right\rangle.
$$
From here we obtain
\begin{align*}
\frac{\partial \abs{u}}{\partial x_j} &= \left\langle \frac{u}{\abs{u}} - v; \frac{\partial u}{\partial x_j} \right\rangle + \left\langle v; \frac{\partial u}{\partial x_j} \right\rangle.
\end{align*}
Since $(I_{3\times 3}-T_v)(v) = v$ and $\abs{v}=1$, we find that
$$
\abs{\nabla_e \abs{u}} \leq 2\frac{\abs{u-v}}{\abs{u}}\abs{\nabla_e u} + \abs{(I_{3\times 3}-T_v)(\nabla_e u ) }.
$$
Further, due to $0 \leq \lambda_3 \leq \lambda_2 \leq 1-\lambda_1$, we have that $\abs{u-v} \leq 3(1-\lambda_1)$.  This and $\abs{u}\geq \lambda_1> 1-r$ lead to the following inequality
\begin{equation}\label{est_modulus}
\abs{\nabla_e \abs{u}}^2 \leq C(1-\lambda_1)\abs{\nabla_e u}^2 + C\abs{(I_{3\times 3}-T_v)(\nabla_e u) }^2,
\end{equation}
where $C>0$ can be chosen independent of $r\in ]0,\frac{1}{4}]$.  We now use (\ref{est_modulus}) and (\ref{est_v}) in (\ref{split_nabla_tan_perp}) to obtain
$$
\abs{\nabla_e u}^2 \geq \abs{\nabla_e v} + \frac{1}{C}\abs{\nabla_e \abs{u}}^2 - C(1-\lambda_1)\abs{\nabla_e u}^2,
$$
or
$$
(1+C(1-\lambda_1))\abs{\nabla_e u}^2 \geq \abs{\nabla_e v}^2 + \frac{1}{C}\abs{\nabla_e \abs{u}}^2.
$$
This implies that
\begin{equation}\label{almost_rad_phase}
\abs{\nabla_e u}^2 \geq (1-C(1-\lambda_1))\abs{\nabla_e v}^2 + \frac{1}{C}\abs{\nabla_e \abs{u}}^2.
\end{equation}
Next, observe that, since the eigenvalues of $u$ are non-negative and add up to $1$, it follows that $\abs{u}\leq 1$ and
$$
1 = (\lambda_1 + \lambda_2 + \lambda_3)^2.
$$
From here we find that
$$
2(1-\abs{u}) \geq 1-\abs{u}^2 = 1-\lambda^2_1-\lambda^2_2-\lambda^2_3 = 2(\lambda_1\lambda_2+\lambda_1\lambda_3+\lambda_2\lambda_3).
$$
This implies that
$$
(1-\abs{u}) \geq \lambda_1(1-\lambda_1) \geq (1-r)(1-\lambda_1),
$$
so
$$
1-\lambda_1 \leq \frac{1-\abs{u}}{1-r}.
$$
Next, let $n \geq 1$ be an integer to be chosen later and write
\begin{align*}
1-C(1-\lambda_1) &= \abs{u}^n + (1-\abs{u}^n) - C(1-\lambda_1) \\
&= \abs{u}^n + (1-\abs{u})\sum_{k=0}^{n-1}\abs{u}^k - C(1-\lambda_1) \\
&\geq \abs{u}^n + (1-\abs{u})\left ( \sum_{k=0}^{n-1}\abs{u}^k - \frac{C}{1-r}\right ) \\
&\geq  \abs{u}^n + (1-\abs{u})\left ( \sum_{k=0}^{n-1}(1-r)^k - \frac{C}{1-r}\right ).
\end{align*}
Now it is clear that we can choose $r > 0$ small enough and $n \geq 1$ large enough so that
$$
\sum_{k=0}^{n-1}(1-r)^k - \frac{C}{1-r} = \frac{1-(1-r)^n}{r} - \frac{C}{1-r}\geq 0.
$$
With such $r>0$ and $n \geq 1$ we obtain
$$
1-C(1-\lambda_1) \geq \abs{u}^n
$$
if ${\rm dist}(u; {\cal P}) < r$.  Going back to (\ref{almost_rad_phase}), we obtain
$$
\abs{\nabla_e u}^2 \geq \abs{u}^n\abs{\nabla_e v}^2 + \frac{1}{C}\abs{\nabla_e \abs{u}}^2.
$$
Finally, we observe that $2(1-\lambda_1) \geq 1-\lambda_1^2\geq 0$, so
\begin{align*}
8W(u) &\geq 4\lambda_1^2(1-\lambda_1)^2 \geq (1-r)^2(1-\lambda_1^2)^2 \\
&\geq (1-r)^2(1-\abs{u}^2)^2 \geq (1-r)^2(1-\abs{u})^2.
\end{align*}
Therefore, if $r>0$ is small enough, ${\rm dist}(u;{\cal P}) < r$, and $n \geq 1$ is large enough, then
\begin{align*}
e_\eps(u) &=  \frac{\abs{\nabla u}^2}{2} + \frac{W(u)}{\eps^2}\\
&\geq \frac{\abs{u}^n \abs{\nabla_\tau Q(u)}^2}{2} + \frac{\abs{\nabla_\tau \abs{u}}^2}{C} + \frac{1}{C\eps^2}\abs{1-\abs{u}}^2.
\end{align*}
The conclusion of the lemma follows since $\rho = \abs{u}$.
\qed
\end{proof}
Next we recall several lemmas that can be proven exactly as in \cite{BBH}.
\begin{lemma}\label{bdd_potential_int}
If $\Omega$ is star-shaped, there exists a constant $C > 0$ independent of $\eps > 0$ such that
$$
\frac{1}{\eps^2} \int_\Omega W(u) + \int_{\partial \Omega} \abs{\nabla u \cdot \nu}^2 \leq C.
$$
\end{lemma}
\begin{proof}
This follows from Pohozaev's identity, as in \cite{BBH}.
\qed
\end{proof}
\begin{lemma}\label{gd_ball_struct}
Let $C > 0$ be such that
\begin{equation}\label{potential_lips}
\abs{\nabla_x W(u)} \leq \frac{C}{\eps}
\end{equation}
in $\Omega$.  For all $0 < r \leq 2C$ there are positive numbers $\lambda_0, \mu_0 > 0$ such that for all $l \geq \lambda_0\eps$ and all $x_0 \in \overline{\Omega}$, 
$$
\frac{1}{\eps^2}\int_{\Omega \cap B_{2l}(x_0)} W(u) \leq \mu_0 \,\,\,\, \Rightarrow \,\,\,\, W(u(x)) \leq r \,\,\,\,\mbox{for all} \,\,\,\, x \in \Omega \cap B_l(x_0).
$$
\end{lemma}
\begin{proof}
Again, the proof of this statement is exactly as in \cite{BBH}.  We pick $x_0 \in \overline{\Omega}$ and assume that there is $y_0 \in B_l(x_0)$ with $W(u(y_0)) \geq r$.  From (\ref{potential_lips}) we obtain
\begin{align*}
W(u(x)) &= W(u(y_0)) + W(u(x)) - W(u(y_0)) \\
&\geq W(u(y_0)) - \frac{C}{\eps}\abs{x - y_0} \\
&\geq r - \frac{C\rho}{\eps},
\end{align*}
for all $x \in B_\rho(y_0)$.  Choose $\rho = \frac{\eps r}{2C}$.  Then,
$$
W(u(x)) \geq \frac{r}{2}
$$
for all $x \in B_\rho(y_0)$.

Observe now that there is a number $\alpha > 0$ such that $\abs{\Omega \cap B_r(x)} \geq \alpha r^2$ for all $x \in \overline{\Omega}$ and all $0 < r \leq 1$. Further, $y_0 \in B_\rho(x_0)$ implies $B_\rho(y_0) \subset B_{2l}(x_0)$, whenever $l \geq \rho = \frac{\eps r}{2C}$. We conclude that
$$
\int_{\Omega \cap B_{2l}(x_0)} W(u) \geq \int_{\Omega \cap B_{\rho}(y_0)} W(u) \geq \frac{r}{2} \alpha \rho^2 = \frac{\alpha r^3 \eps^2}{4C^2}.
$$
Set $\lambda_0 = \frac{r}{2C}$ and $0 < \mu_0 < \frac{\alpha r^3}{4C^2}$.  This proves the lemma.
\qed
\end{proof}

We assume now that $g: \partial \Omega \to \cal P$ represents a non-contractible curve in $\cal P$.  Recall the definition of the smooth open set $\Omega_L$ given in the first paragraph of this section. In particular, we may consider $u_\eps$ to be defined in $\Omega_L$, but independent of $\eps$ in $\Omega_L\setminus \Omega$.  Our next lemma is the following
\begin{lemma}\label{lemma_one_sing}
Let $u_\eps \in W^{1,2}(\Omega;F_1)$ be a minimizer of $E_\eps$ among $u \in W^{1,2}(\Omega;F_1)$ such that $u=g$ on $\partial \Omega$.  There is a single $a \in \overline{\Omega}$ with the following property: there is a constant $C > 0$ such that, for every $R > 0$ there is an $\eps_0> 0$ such  that, for every $0 < \eps \leq \eps_0$ we have
$$
\int_{\Omega_L \setminus B_R(a)}e_\eps(u_\eps) \leq C.
$$
\end{lemma}

\begin{proof}
We start with the following simple observation: for any $b \in \Omega$ and $r_1 > 0$ such that $B_{2r_1}(b) \subset \Omega$, one can build a function $v_\eps \in W^{1,2}(\Omega;F_1)$ such that $v_\eps = g$ on $\partial \Omega$, $v_\eps$ equal $g_0$ on $\partial B_{r_1}(b)$, where $g_0$ is a (fixed) closed geodesic in $\cal P$ appropriately parametrized and
\begin{equation}\label{easy_upper_bound}
E_\eps(u_\eps) \leq E_\eps(v_\eps) \leq \frac{L_0^2}{4\pi}\ln\left ( \frac{1}{\eps}\right ) + C,
\end{equation}
where $C > 0$ is a constant that depends on $b \in \Omega$ and $r_1>0$, but is independent of $\eps$.

We will show next that there is a single $a \in \overline{\Omega}$ with the following property: there is a constant $C>0$ such that, for any $R>0$ with $B_R(a) \subset \Omega_L$, there is $\eps_0>0$ such that, for all $0<\eps\leq \eps_0$ we have
\begin{equation}\label{lower_ball}
\int_{B_R(a)}e_\eps(u_\eps) \geq \frac{L_0^2}{4\pi}\ln\left (\frac{R}{\eps} \right) -C.
\end{equation}
These two observations will lead to the conclusion of the lemma.

The proof of (\ref{lower_ball}) is a combination of arguments from \cite{BBH}, \cite{jerrard_lower} and \cite{sandier}, that we can use because of \lemref{Jerrard_form}.

We argue first in the following manner as in \cite{BBH}

Let $r > 0$ be such that the hypotheses of \lemref{Jerrard_form} are satisfied and such that $Q(u)$ is well defined for every $u \in \Sigma$ with $W(u) < r$.  Using this $r$, choose $\lambda_0, \mu_0 > 0$ as in \lemref{gd_ball_struct}.  We can select a collection of points ${\cal C} = \{x_i\}_{i \in I} \subset \Omega$ such that
\begin{enumerate}
\item $\overline{\Omega} \subset \bigcup_{i \in I}B_{\lambda_0 \eps}(x_i)$
\item $B_{\frac{\lambda_0 \eps}{2}}(x_i) \cap B_{\frac{\lambda_0 \eps}{2}}(x_j) = \emptyset$ whenever $i,j \in I$, $i \neq j$.
\end{enumerate}
Observe that the second condition, plus simple geometry demonstrate the existence of a non-negative integer $\iota$ with the following property: every $x \in \overline{\Omega}$ has
$$
{\rm card}(\{j \in I : x\in B_{2\lambda_0 \eps}(x_j)\}) \leq \iota.
$$
From here we define
$$
J_\eps = \{i \in I : \frac{1}{\eps^2} \int_{B_{2\rho_\eps}(x_i)} W(u) > \mu_0\},
$$
where $\rho_\eps = \lambda_0 \eps$.  We find that there is a natural number $N$ independent of $\eps>0$ such that
$$
{\rm card}(J_\eps) \leq N.
$$
Next we iteratively build finite families of balls, ${\cal C}_i = \{B_{r^i_j}(x^i_j)\}_{j=1}^{N(i)}$, that will contain the main part of the energy $E_\eps(u_\eps; B_R(a))$.  This construction follows very closely the Jerrard/Sandier arguments from \cite{jerrard_lower} and \cite{sandier} and starts with ${\cal C}_0 = \{B_{r_i}(x_i): i \in J_\eps\}$.  Here $r_i = \rho_\eps = \lambda_0\eps$ for all $i\in J_\eps$.  We then use a merger argument as follows:  if $i,j \in J_\eps$, $i \neq j$ are such that
$$
\abs{x_i - x_j} \leq r_i + r_j,
$$
we replace the balls $B_{r_i}(x_i)$, $B_{r_j}(x_j)$ by a single ball centered at
$$
x = \frac{r_i}{r_i+r_j}x_i + +\frac{r_j}{r_i+r_j}x_j
$$
with radius $r = r_i + r_j$.  It is straightforward to check both that the original balls are contained in the new one and that we forced the radius of the new ball to be the sum of the radii of the original balls.  We continue this procedure until we have a family of balls $\{B_{r_i}(x_i)\}_{i \in J}$  such that
\begin{enumerate}
\item $\abs{J} \leq N$
\item $r_i \geq \rho_\eps$ for all $i \in J_\eps$, $\sum_{i \in J} r_i \leq N\lambda_0 \eps$
\item $\abs{x_i - x_j} > r_i + r_j$ for all $i,j \in J$, $i \neq j$.
\end{enumerate}
Observe that we have
$$
W(u(x)) \leq r,
$$
for all $x \in \Omega \setminus \bigcup_{i \in J} B_{r_i}(x_i)$.

Let us now denote by $a_1,...,a_m$ the distinct limits of the $\{x_i\}_{i \in J}$ as $\eps \to 0$ and choose $R > 0$ such that
$$
B_{2R}(a_j) \cap B_{2R}(a_k) = \emptyset, \,\,\,\, B_{2R}(a_j) \subset \Omega_L
$$
for all $1 \leq j,k \leq m$, $j\neq k$.  Let also $\eps_0 >0$ be small enough so that, for all $0 < \eps \leq \eps_0$, every $x_j \in B_{R/4}(a_k)$ for some $k = 1,...,m$.  Note that $Q(u_\eps)$ is well-defined on each $\partial B_R(a_j)$.

Suppose that $a_j$ is such that $Q(u_\eps)$ is non-contractible on $\partial B_R(a_j)$.  Denote $a_j$ by $a$ and set
$$
J_1=\{j : B_{r_j}(x_j)\subset B_R(a)\}.
$$
Possibly by relabeling, we assume that $J_1=\{1, ..., k(1)\}$ and write $x_j^1$ and $\rho_j^1$ instead of $x_j$ and $r_j$, respectively.  Observe that so far we know that
$$
\sum_{j \in J_1} \rho_j^1 \leq \lambda_0 N \eps \,\,\,\,\mbox{and}\,\,\,\, i,j \in J_1, i \neq j \implies \abs{x^1_j-x^1_i} > \rho^1_j + \rho^1_i.
$$
Let
\begin{align*}
\delta^1_j = \left \{ 
\begin{array}{cc} 
1, & \mbox{if }Q(u_\eps)\mbox{ is non-contractible on }\partial B_{\rho^1_j}(x^1_j),\\
0, &\mbox{otherwise.} 
\end{array}
\right.
\end{align*}
Since $Q(u_\eps)$ is non-contractible on $\partial B_R(a)$, at least one $\delta_j^1 = 1$.  Define also
\begin{align*}
t_1 = \sup \left\{t>0 : B_{\rho^1_j+\delta^1_jt}(x^1_j) \subset \Omega_L, \, j=1,...,k, \,\,\,\, \mbox{and} \right.\\
\left.B_{\rho^1_j+\delta^1_jt}(x^1_j) \cap B_{\rho^1_i+\delta^1_it}(x^1_i) = \emptyset \,\,\,\mbox{if}\,\,\, \,i \neq j\right\}.
\end{align*}
Since at least one $\delta^1_j = 1$, then $0 < t_1 < +\infty$. There are two mutually exclusive options for $t_1$: 
\begin{enumerate}
\item There is a $j \in \{1,..,k(1)\}$ such that $B_{\rho^1_j+\delta^1_jt_1}(x^1_j)$ touches the boundary of $\Omega_L$,
\item Two or more balls $B_{\rho^1_j+\delta^1_jt_1}(x^1_j)$ touch each other without either touching $\partial \Omega_L$.
\end{enumerate}
In the first case the procedure terminates.  If, for example, $x_j^1$ is the point for which $B_{r_j^1+t_1\delta_j^1}(x_j^1)$ touches $\partial \Omega_L$, we must have $r_j^1+\delta_j^1t_1 = {\rm dist}(x^1_j, \Omega_L)$.  By the choice of $R>0$, we have that
$$
2R\leq {\rm dist}(a, \Omega_L) \leq \abs{x^1_j-a}+{\rm dist}(x^1_j, \Omega_L) \leq \frac{R}{4}+r_j^1+\delta_j^1t_1.
$$
Hence $r_j^1+\delta_j^1t_1\geq \frac{7R}{4}$ and $\delta_j^1 = 1$.  Now we also have that $B_{\frac{R}{4}}(a)\subset B_{\rho_j^1+t_1}(x_j^1)$ and therefore $k(1)=1$. That is, there is only one $x^1_j$.

In the second case two or more balls touch each other.  Set
$$
\rho_j^2 = \rho_j^1+\delta_j^1t_1, \,\,\,\, x_j^2 = x_j^1.
$$
Again, replace each pair of balls that touch, say $B_{\rho_i^2}(x_i^2)$ and $B_{\rho_j^2}(x_j^2)$, with a single ball with radius equal to the sum of the radii of the original balls  and the center at
$$
x = \frac{\rho_j^2}{\rho_j^2+\rho_j^2}x_j^2 + \frac{\rho_i^2}{\rho_j^2+\rho_j^2}x_i^2.
$$
Observe that the new ball contains both balls $B_{\rho_i^2}(x_i^2)$ and $B_{\rho_j^2}(x_j^2)$.  Repeat this procedure until we arrive at a set of balls with disjoint closures.  With a slight abuse notation, denote the centers and radii of these balls by $x_j^2$ and $\rho_j^2$, $j=1,...,k(2)$, respectively.  Set
$$
J_j^2 = \{i : B_{\rho_i^1+t_1\delta_i^1}(x_i^1) \subset B_{\rho_j^2}(x_j^2)\},
$$
and observe that
$$
\bigcup_{i \in J_j^2} B_{\rho_i^1+t_1\delta_i^1}(x_i^1) \subset B_{\rho_j^2}(x_j^2) \,\,\,\,\mbox{and}\,\,\,\, \rho_j^2 = \sum_{i \in J_j^2} (\rho_i^1 + t_1\delta_j^1).
$$
We point out that the $x_j^2$ are in the convex envelope of the $x_j^1$ and hence they are always in $B_{R/4}(a)$.  In particular, if $k(2) \geq 2$, then $\rho_j^2 \leq R/2$ for all $j=1,...,k(2)$.  We iterate this procedure until one of the balls touches $\partial \Omega_L$.  Suppose that this occurs at the step $M$.  At this point, for each $1 \leq m \leq M$ we have the following.
\begin{enumerate}
\item There is a collection of points $x_j^m \in \overline{\Omega}$, integers $\delta_j^m \in \{0,1\}$, and real numbers $t_m, \rho_j^m > 0$, $j=1,...,k(m)$ with $x_j^m \in B_{R/4}(a)$ such that
$$
{\rm meas}(B_{\rho_j^m+t_m\delta_j^m}(x_j^m) \cap B_{\rho_i^m+t_m\delta_i^m}(x_i^m)) = 0,
$$
if $1 \leq i < j\leq k(m).$
\item If $m \leq M-1$, for every $j \in \{1,...,k(m)\}$ there is a $k \in \{1,...,k(m+1)\}$ such that $B_{\rho_j^m+t_m\delta_j^m}(x_j^m) \subset B_{\rho_k^{m+1}}(x_k^{m+1}).$
\item $\displaystyle \rho_k^m = \sum_{i \in J_k^m} (\rho_i^{m-1} + t_{m-1}\delta_i^{m-1}),$
where $J_k^m = \{i: B_{\rho_i^{m-1}}(x_i^{m-1}) \subset B_{\rho_k^{m}}(x_k^{m})\},$ for $m \geq 2.$
\item There is at least one $i \in J_j^m$ such that $\delta_i^{m-1} = 1$ if $m \geq 2$ and $\delta_m^j = 1.$
\item $\displaystyle \sum_{j=1}^{k(1)} \rho_1^1 \leq \lambda_0 N \eps.$
\item $k(M) = 1$, $\rho_1^M + t_M \geq R$, $\delta_1^M = 1$ and $B_{\rho_1^M+t_M}(x_1^M) \subset \Omega_L.$
\end{enumerate}

Once the sets ${\cal C}_i$ are built, we need to estimate the integral of $e_\eps(u)$ over ${\cal C}_i$.  To this end, recall the following definition from \cite{jerrard_lower}:
\begin{equation}\label{little_lambda_jerrard}
\lambda_\eps(s) = \min_{m \in [0,1]} \left ( \frac{L_0^2m^n}{4\pi s} + \frac{1}{C\eps}(1-m)^N \right ).
\end{equation}
We observe that \lemref{Jerrard_form} and Theorem 2.1 in \cite{jerrard_lower} imply that there are constants $C,N>0$ such that, whenever $Q(u_\eps)$ is non-contractible on the circle $\partial B_s(x)$ (this is the case 
when $W(u_\eps) < r$ on $\partial B_s(x)$ and $r>0$ is as defined in \lemref{Jerrard_form}), then for $s > \eps$ it follows that 
\begin{align*}
\int_{\partial B_s(a)} e_\eps(u) &\geq \int_{\partial B_s(a)} \left ( \frac{\rho^n \abs{\nabla_\tau Q(u)}^2}{2} + \frac{\abs{\nabla_\tau \rho}^2}{C} + \frac{1}{C\eps^2}\abs{1-\rho}^2\right ).
\end{align*}
Next we define $m_s = \min\{\abs{u(x)}: x\in \partial B_s(a)\}$ and use Lemma 2.3 from \cite{jerrard_lower} to obtain
\begin{align*}
\int_{\partial B_s(a)} e_\eps(u) &\geq \frac{m_s^n}{2}\int_{\partial B_s(a)} \abs{\nabla_\tau Q(u)}^2\,dl + \frac{1}{C\eps}\abs{1-m_s}^M \\
&\geq \frac{m_s^nL_0^2}{4\pi s} + \frac{1}{C\eps}\abs{1-m_s}^M,
\end{align*}
for some $M > 1$.  By definition, 
$$
\int_{\partial B_s(x)}e_\eps(u_\eps) \geq \lambda_\eps(s).
$$
Furthermore, also from \cite{jerrard_lower}, we have
$$
\lambda_\eps(s) \geq \frac{L_0^2}{4\pi s}\left (1 - C\frac{\eps^\alpha}{s^\alpha}\right ),
$$
for some constants $C,\alpha > 0$ that do not depend on $s, \eps$.  This shows that when both $Q(u_\eps)$ is well-defined and non-contractible and $W(u_\eps)<r$ in an annulus $B_{s_1}\setminus B_{s_0}(x)$, where $s_0 > \eps$ then
\begin{equation}\label{lower_annulus}
\int_{B_{s_1}\setminus B_{s_0}(x)} e_\eps \geq  \frac{L_0^2}{4\pi}\ln\left ( \frac{s_1}{s_0}\right ) - C,
\end{equation}
for some constant $C > 0$ independent of $\eps$ and independent of $s_0, s_1 \in ]\eps,{\rm diam}(\Omega_L)]$.

Finally we compute:
\begin{align*}
\int_{B_{\rho_1^M+t_M}(x_1^M)} e_\eps(u_\eps) &= \int_{(B_{\rho_1^M+t_M}\setminus B_{\rho_1^M})(x_1^M)} e_\eps(u_\eps)+\int_{B_{\rho_1^M}(x_1^M)} e_\eps(u_\eps)  \\
&\geq  \frac{L_0^2}{4\pi}\ln\left ( \frac{\rho_1^M+t_M}{\rho_1^M}\right ) + \sum_{j \in J_1^M} \int_{(B_{\rho_j^{M-1}+t_{M-1}\delta_j^{M-1}}\setminus B_{\rho_j^{M-1}})(x_1^M)} e_\eps(u_\eps)\\
&+ \sum_{j \in J_1^M}\int_{B_{\rho_j^{M-1}}(x_1^{M-1})} e_\eps(u_\eps) -C\\
&\geq \frac{L_0^2}{4\pi} \left (\ln\left ( \frac{\rho_1^M+t_M}{\rho_1^M}\right ) + \sum_{j \in J_1^M} \ln\left ( \frac{\rho_j^{M-1}+t_{M-1}\delta_j^{M-1}}{\rho_j^{M-1}}\right ) \right )\\
&+ \sum_{j \in J_1^M}\int_{B_{\rho_j^{M-1}}(x_j^{M-1})} e_\eps(u_\eps) -C\\
&\geq \frac{L_0^2}{4\pi} \left (\ln\left ( \frac{\rho_1^M+t_M}{\rho_1^M}\right ) +  \ln\left ( 1 + \frac{ \sum_{j \in J_1^M} t_{M-1}\delta_j^{M-1}}{\sum_{j \in J_M} \rho_j^{M-1}}\right ) \right )\\
&+ \sum_{j \in J_1^M}\int_{B_{\rho_j^{M-1}}(x_j^{M-1})} e_\eps(u_\eps) -C \\
&\geq \frac{L_0^2}{4\pi} \ln\left ( \frac{\rho_1^M+t_M}{\sum_{j \in J_1^M} \rho_j^{M-1}}\right ) + \sum_{j \in J_1^M}\int_{B_{\rho_j^{M-1}}(x_j^{M-1})} e_\eps(u_\eps) -C
\end{align*}
At this point we iterate to finally arrive at
\begin{align*}
\int_{B_{\rho_1^M+t_M}(x_1^M)} e_\eps(u_\eps) \geq \frac{L_0^2}{4\pi} \ln\left ( \frac{\rho_1^M+t_M}{\sum_{j =1}^{k(1)} \rho_j^{1}}\right ) -C.
\end{align*}
Since $\sum_{j =1}^{k(1)} \rho_j^{1} \leq \lambda_0N \eps$ and $\rho_1^M + t_M \geq R$, the conclusion of the lemma follows. 
\qed
\end{proof} 

\begin{remark}\label{remark_upper_outer_ball}
Because of (\ref{easy_upper_bound}) and (\ref{lower_ball}), for the point $a$ in the previous theorem and $R>0$ such that $B_R(a)\subset \subset \Omega_L$ and $Q(u)$ is non-contractible on $\partial B_R(a)$, we have
\begin{equation}\label{upper_outer_ball}
\int_{\Omega_L \setminus B_R(a)}e_\eps(u_\eps) \leq \frac{L_0^2}{4\pi}\ln\left (\frac{1}{R} \right) +C,
\end{equation}
for a constant $C > 0$ independent of $\eps>0$ and $R>0$ with $B_R(a) \subset \subset \Omega_L$.
\end{remark}

From this last inequality we conclude that $u_\eps$ are bounded in $W^{1,2}(\Omega_L\setminus B_R(a); F_1)$ for any fixed $R > 0$, where $a\in \Omega_L$ is the point from \lemref{one_sing}.  By a standard diagonalization argument we then obtain the existence of $u_0 \in W^{1,2}_{loc}(\Omega_L \setminus \{a\}; F_1)$ such that, along a subsequence,
$$
u_\eps \rightharpoonup u_0
$$
in $W^{1,2}(\Omega_L\setminus B_R(a); F_1)$ for any fixed $R > 0$.  We will prove the
\begin{lemma}\label{convergence_strong_away_a}
Along a subsequence, we have that
$$
u_\eps \to u_0
$$
in $W^{1,2}(\Omega_L\setminus B_R(a); F_1)$ for any fixed $R > 0$ .
\end{lemma}
\begin{proof}
Let first $x \in \Omega\setminus \{a\}$, and $r > 0$ such that $a \notin \overline{B_{2r}(x)}\subset \Omega$.  We know that
\begin{equation}\label{u_eps_bdd_on_bdry_ball}
\int_{B_{2r}(x)\setminus B_r(x)}e_\eps(u_\eps) \leq C,
\end{equation}
for some constant independent of $\eps, r>0$.  We also know that, along a subsequence,
$$
\int_{B_{2r}(x)}\abs{u_\eps - u_0}^2 \to 0.
$$
By Fatou's Lemma and Fubini's Theorem there is a $\rho \in [r,2r]$ such that
$$
\int_{\partial B_\rho(x)}e_\eps(u_\eps)\leq C
$$
and
$$
\int_{\partial B_{\rho}(x)}\abs{u_\eps - u_0}^2 \to 0,
$$
along some subsequence $\eps_n \to 0$.  Dropping the index $n$ for simplicity, we observe that the $u_\eps$ are uniformly H\"older continuous on $\partial B_\rho(x)$, because the integrals $\int_{\partial B_\rho(x)} \abs{\nabla u_\eps}^2$ are uniformly bounded.  In particular, along a subsequence, $u_\eps \to u_0$ uniformly on $\partial B_\rho(x)$.  By Remark (\ref{remark_upper_outer_ball}), the map $Q(u_\eps)$ must be contractible on $\partial B_\rho(x)$.  Since
$$
\abs{Q(u_\eps)-u_0}Ê\leq \abs{Q(u_\eps)-u_\eps}+\abs{u_\eps-u_0} \leq 2\abs{u_\eps - u_0},
$$
the function $u_0$ must also be continuous and contractible on $\partial B_\rho(x)$.

Next, define $Z_\eps : B_\rho(x) \to F_0$ by
\begin{align}
-\Delta Z_\eps + \frac{1}{\eps^2}Z_\eps = 0 \,\,\,\, &\mbox{in} \,\,\,\, B_\rho(x) \nonumber\\
Z_\eps = u_\eps - Q(u_\eps) \,\,\,\, &\mbox{on} \,\,\,\, \partial B_\rho(x).
\end{align}
We have the estimate
$$
\int_{B_\rho(x)} \left \{\abs{\nabla Z_\eps}^2 + \frac{1}{\eps^2} \abs{Z_\eps}^2 \right \} \leq C\eps.
$$
As in \cite{BBH}, this follows from Pohozaev's identity applied to $Z_\eps$.  Now let $R_\eps$ satisfy
\begin{align}
-\Delta R_\eps = 0 \,\,\,\, &\mbox{in} \,\,\,\, B_\rho(x) \nonumber\\
R_\eps = R(u_0, Q(u_\eps)) \,\,\,\, &\mbox{on} \,\,\,\, \partial B_\rho(x),
\end{align}
where $R(P,Q)$ is as defined in (\ref{def_min_rotation}).  We verify that $R_\eps \to I_3$ strongly in $W^{1,2}(B_\rho(x); M^3(\R{}))$.  Indeed, observe first that $Q(u_\eps) \to u_0$ uniformly on $\partial B_\rho(x)$, hence $R_\eps \to I_3$ uniformly on $\partial B_\rho(x)$.  Because $R_\eps - I_3$ is harmonic in $B_\rho(x)$, it follows that $R_\eps \to I_3$ uniformly in $B_\rho(x)$.  Next we find that
$$
\int_{B_\rho(x)}\abs{\nabla R_\eps}^2 = \int_{\partial B_\rho(x)} \langle R_\eps ; \nabla R_\eps \cdot \nu \rangle = \int_{\partial B_\rho(x)} \langle (R_\eps -I_3); \nabla R_\eps \cdot \nu \rangle,
$$
where we used the fact that $R_\eps$ is harmonic.  Then
$$
\int_{B_\rho(x)}\abs{\nabla R_\eps}^2 \leq \left ( \int_{\partial B_\rho(x)} \abs{R_\eps - I_3}^2 \int_{\partial B_\rho(x)} \abs{\nabla R_\eps \cdot \nu}^2 \right )^{\frac{1}{2}}.
$$
However, a harmonic function on the disk $B_\rho(x)$ satisfies the classical equipartition of the energy property, that is,
$$
\int_{\partial B_\rho(x)} \abs{\nabla R_\eps \cdot \nu}^2 = \int_{\partial B_\rho(x)} \abs{\nabla R_\eps \cdot \tau}^2.
$$
We then obtain 
$$
\int_{B_\rho(x)}\abs{\nabla R_\eps}^2 \leq \left ( \int_{\partial B_\rho(x)} \abs{R_\eps - I_3}^2 \int_{\partial B_\rho(x)} \abs{\nabla R_\eps \cdot \tau}^2 \right )^{\frac{1}{2}}.
$$
Thus, $R_\eps = R(u_0,Q(u_\eps))$ on $\partial B_\rho(x)$.  \lemref{expr_min_rot} and (\ref{u_eps_bdd_on_bdry_ball}) show that $\int_{\partial B_\rho(x)} \abs{\nabla R_\eps \cdot \tau}^2$ is uniformly bounded in $\eps > 0$.  Since $Q(u_\eps) \to u_0$ uniformly, we obtain that $\int_{B_\rho(x)}\abs{\nabla R_\eps}^2 \to 0$ and $R_\eps \to I_3$ strongly in $W^{1,2}(B_\rho(x); M^3(\R{}))$.

We now define
$$
v_\eps = Q \left (R_\eps u_0 R_\eps^T \right ) + Z_\eps.
$$
Note that $W(v_\eps) \leq C\abs{Z_\eps}^2$.  It also follows easily from the previous discussion that $v_\eps \to u_0$ strongly in $W^{1,2}(B_\rho(x);M^3(\R{}))$.  Therefore
$$
\lim_{\eps \to 0} E_\eps(v_\eps;B_\rho(x)) = \int_{B_\rho(x)}\frac{\abs{\nabla u_0}^2}{2},
$$
and we have
$$
v_\eps = \pi \left (R_\eps u_0 R_\eps^T \right ) + Z_\eps = Q(u_\eps) + u_\eps - Q(u_\eps) = u_\eps.
$$
on $\partial B_\rho(x)$. This shows that
$$
E_\eps(u_\eps; B_\rho(x)) \leq E_\eps(v_\eps; B_\rho(x)).
$$
From here we deduce
\begin{align*}
\frac{1}{2} \int_{B_\rho(x)} \abs{\nabla u_0}^2 &\leq \liminf_{\eps \to 0} \frac{1}{2} \int_{B_\rho(x)} \abs{\nabla u_\eps}^2 \\
&\leq  \limsup_{\eps \to 0} \frac{1}{2} \int_{B_\rho(x)} \abs{\nabla u_\eps}^2 \\
&\leq \limsup_{\eps \to 0} E_\eps(u_\eps; B_\rho(x)) \\
&\leq \limsup_{\eps \to 0} E_\eps(v_\eps; B_\rho(x)) = \frac{1}{2} \int_{B_\rho(x)} \abs{\nabla u_0}^2.
\end{align*}
In other words, 
$$
\frac{1}{2} \int_{B_\rho(x)} \abs{\nabla u_0}^2 = \lim_{\eps \to 0} \frac{1}{2} \int_{B_\rho(x)} \abs{\nabla u_\eps}^2.
$$
Thus $u_\eps \to u_0$ strongly in $W^{1,2}(B_r(x);M^3(\R{}))$ when $B_{2r}(x) \subset \Omega$.  However, this argument also works with small modifications when $x \in \partial \Omega$ under he assumption that $B_r(x) \cap \Omega$ is strictly starshaped (with respect to a point in the interior of $B_r(x) \cap \Omega$).  This shows that, for fixed $R > 0$, the sequence $u_\eps \to u_0$ strongly in $W^{1,2}(\Omega \setminus B_R(a);M^3(\R{}))$.
\qed
\end{proof}
\begin{remark}
Once we know the conclusion of the last Lemma, a simple modification of the argument in its proof shows the following:  if $U \subset \subset \overline{\Omega}\setminus \{a\}$ is an open set, and $v \in W^{1,2}_{loc}(\Omega\setminus \{a\}; {\cal P})$ has $v = u_0$ on $\partial U$, then
$$
\int_U\abs{\nabla u_0}^2 \leq \int_U\abs{\nabla v}^2.
$$
\end{remark}

The only claim of \thmref{one_sing} that still needs to be proved is the
\begin{lemma}\label{lemma_sing_inside}
Let $a \in \overline{\Omega}$ be as in \lemref{lemma_one_sing}. Then $a$ is in the interior of $\Omega$.
\end{lemma}
\begin{proof}
The proof of this lemma is along the lines of a similar result in in \cite{BBH}.  First, assume $a \in \partial \Omega$ and observe that, for almost every $r > 0$, the sequence $u_\eps \to u_0$ strongly in $W^{1,2}(\partial B_r(a);\cal P)$.  Hence, for almost every $r > 0$, the function $u_0$ is continuous on $\partial B_r(a)$ and non-contractible on $\partial B_r(a)$.  It follows that
$$
L_0 \leq \int_{\partial B_r(a)}\abs{\nabla u_0} = \int_{\partial B_r(a)\cap \Omega}\abs{\nabla u_0} + \int_{\partial B_r(a)\setminus \Omega}\abs{\nabla u_0}.
$$
However, $u_0$ is smooth outside $\overline{\Omega}$.  Hence there is a constant $C > 0$, independent of $r>0$, such that
$$
L_0 \leq \int_{\partial B_r(a)\cap \Omega}\abs{\nabla u_0} + CH^{(1)}(\partial B_r(a)\setminus \Omega).
$$
From here we obtain
$$
L_0^2 - CH^{(1)}(\partial B_r(a)\setminus \Omega) \leq H^{(1)}(\partial B_r(a)\cap \Omega)\int_{\partial B_r(a)\cap \Omega}\abs{\nabla u_0}^2.
$$
Then
\begin{equation}\label{est_half_ball}
\frac{L_0^2}{H^{(1)}(\partial B_r(a)\cap \Omega)}-C\frac{H^{(1)}(\partial B_r(a)\setminus \Omega)}{H^{(1)}(\partial B_r(a)\cap \Omega)}\leq \int_{\partial B_r(a)\cap \Omega}\abs{\nabla u_0}^2.
\end{equation}
Since $\Omega$ is smooth, there are constants $r_0> 0$ and $\alpha > 0$ such that
$$
\frac{3\pi}{2}\geq H^{(1)}(\partial B_r(a)\cap \Omega) \,\,\,\,\mbox{and}\,\,\,\, \frac{H^{(1)}(\partial B_r(a)\setminus \Omega)}{H^{(1)}(\partial B_r(a)\cap \Omega)}\leq \alpha
$$
for all $r \in ]0,r_0]$.  Integrating (\ref{est_half_ball}) over $[\eta,r_0]$, we obtain
$$
\frac{2L_0^2}{3\pi}\ln\left(\frac{r_0}{\eta}\right ) - C \leq \int_{\Omega \cap (B_{r_0}(a)\setminus B_\eta(a)}\abs{\nabla u_0}^2.
$$
For $\eta>0$ sufficiently small this contradicts (\ref{upper_outer_ball}), because $u_\eps \rightharpoonup u_0$ in $W^{1,2}(B_{r_0}(a)\setminus B_\eta(a);F_1)$.  It follows that $a$ is in the interior of $\Omega$.
\qed
\end{proof}

\section{Proof of \thmref{main_theo}}

We now turn our attention to the proof of \thmref{main_theo}:
\begin{proof}
We will divide the proof into four steps.  Throughout this proof $r$ will denote a positive number, and $r_a = r_a(x) = \abs{x-a}$ will denote the distance from $x$ to $a$.

{\em Step 1---Basic properties of $u_{0}$}. Because of the minimizing property of $u_0$ stated in \thmref{one_sing}, this map satisfies the equation
$$
-\Delta u_0 = 2(\abs{\nabla u_0}^2 u_0 - (u_{0,x}^2+ u_{0,y}^2) )
$$
in $\Omega \setminus \{a\}$.  Here $u_{0,x}, u_{0,y}$ denote the derivatives of $u_0$, and $u_{0,x}^2$ denotes the matrix $u_{0,x}$ multiplied by itself.  Observe that the right hand side of the equation that $u_0$ satisfies is normal to the tangent to $\cal P$ at $u_0(x)$ and that it commutes with $u_0$.  Let us define
$$
v(x) = \nabla u_0\cdot F,
$$
where $F = (f_1,f_2)$ is any vector field (with real entries) in $\R{2}$.  Taking the scalar product of $v$ with both sides of the equation satisfied by $u_0$, we obtain
$$
v(x) \cdot \Delta u_0 = 0.
$$
Following the standard Pohozaev trick we further obtain 
$$
{\rm div}\left ( \left ( \frac{\abs{\nabla u_0}^2}{2} I_2 - Du_0^TDu_0 \right )F \right ) = \left ( \frac{\abs{\nabla u_0}^2}{2} I_2 - Du_0^TDu_0 \right )\cdot DF.
$$
Here
$$
Du_0^TDu_0 = \left ( \begin{array}{cc} u_{0,x} \cdot u_{0,x} & u_{0,x} \cdot u_{0,y} \\ u_{0,x} \cdot u_{0,y} & u_{0,y} \cdot u_{0,y} \end{array} \right )
$$
is the first fundamental form of $u_0$, and $DF$ is the Jacobian matrix of $F$.  We choose now $0<r_0 \leq r_1$ with $B_{r_1}(a) \subset \subset \Omega$, Let $F(x) = x-a$, and integrate this last equation over $B_{r_1}(a) \setminus B_{r_0}(a)$.  We get
$$
r_1 \int_{\partial B_{r_1}(a)} \left ( \frac{\abs{\nabla u_0}^2}{2} - \abs{\nabla u_0 \cdot \nu}^2\right ) = r_0 \int_{\partial B_{r_0}(a)} \left ( \frac{\abs{\nabla u_0}^2}{2} - \abs{\nabla u_0 \cdot \nu}^2\right ).
$$
It then follows that the function
$$
\xi(r) = r \int_{\partial B_{r}(a)} \left ( \frac{\abs{\nabla u_0}^2}{2} - \abs{\nabla u_0 \cdot \nu}^2\right )
$$
is constant.  Denote this constant by $\lambda/2$, so that
\begin{equation}\label{poho_cons}
\int_{\partial B_{r}(a)} \abs{\nabla u_0 \cdot \tau}^2 = \frac{\lambda}{r} + \int_{\partial B_{r}(a)} \abs{\nabla u_0 \cdot \nu}^2.
\end{equation}
Integrating this last identity over $[r_0, r_1]$ we obtain
$$
\int_{B_{r_1}(a) \setminus B_{r_0}(a)} \abs{\nabla u_0}^2 = \lambda \ln\left ( \frac{r_1}{r_0}\right ) + 2\int_{B_{r_1}(a) \setminus B_{r_0}(a)} \abs{\nabla u_0 \cdot \nu}^2 \geq \lambda \ln\left ( \frac{r_1}{r_0}\right ).
$$
Considering (\ref{upper_outer_ball}), the fact that $u_\eps \rightharpoonup u_0$, and letting $r_0 \to 0$, we conclude that
$$
\lambda \leq \frac{L_0^2}{2\pi}.
$$
In fact, (\ref{upper_outer_ball}) gives us
\begin{align*}
\frac{L_0^2}{2\pi} \ln\left ( \frac{1}{r_0}\right ) +C &\geq \int_{B_{r_1}(a) \setminus B_{r_0}(a)} \abs{\nabla u_0}^2 \\ &= \int_{B_{r_1}(a) \setminus B_{r_0}(a)} \abs{\nabla u_0 \cdot \tau}^2 + \int_{B_{r_1}(a) \setminus B_{r_0}(a)} \abs{\nabla u_0 \cdot \nu}^2.
\end{align*}
Observing that
\begin{align*}
\int_{\partial B_r(a)}\abs{\nabla u_0\cdot \tau}^2 \geq \frac{1}{2\pi r}\left ( \int_{\partial B_r(a)}\abs{\nabla u_0\cdot \tau} \right )^2 \geq \frac{L_0^2}{2\pi r},
\end{align*}
we conclude that
$$
\int_{B_{r_1}(a) \setminus B_{r_0}(a)}\abs{\nabla u_0\cdot \nu}^2 \leq C
$$
for any $0 < r_0 \leq r_1$, where $C = C(r_1)$.  In light of (\ref{upper_outer_ball}) we then have
\begin{equation}\label{normal_part_integrable}
\int_{\Omega}\abs{\nabla u_0\cdot \nu}^2 < +\infty.
\end{equation}
Next we integrate (\ref{poho_cons}) with respect to $r$ over $[r_0, r_1]$ to obtain
$$
\frac{L_0^2}{2\pi} \ln\left ( \frac{r_1}{r_0}\right ) \leq \int_{B_{r_1}(a) \setminus B_{r_0}(a)}\abs{\nabla u_0\cdot \tau}^2 = \lambda \ln\left ( \frac{r_1}{r_0}\right )  +\int_{B_{r_1}(a) \setminus B_{r_0}(a)}\abs{\nabla u_0\cdot \nu}^2.
$$
Letting $r_0 \to 0$ in the last equation, we see through (\ref{normal_part_integrable}) that
$$
\lambda \geq \frac{L_0^2}{2\pi}.
$$
We conclude that
\begin{equation}\label{right_poho_constant}
\lambda = \frac{L_0^2}{2\pi}.
\end{equation}
Let now $r>0$ and pick $\rho = \rho(r) \in [\frac{r}{2},r]$ such that
\begin{equation}\label{choice_rho}
\int_{\partial B_{\rho(r)}(a)}\abs{\nabla u_0 \cdot \nu}^2 \leq \frac{2}{r}\int_{B_r(a) \setminus B_{\frac{r}{2}}(a)}\abs{\nabla u_0 \cdot \nu}^2.
\end{equation}
Since $\int_{\Omega}\abs{\nabla u_0\cdot \nu}^2 < +\infty$, we conclude that
$$
\lim_{r \to 0} \left ( \rho(r)\int_{\partial B_{\rho(r)}(a)}\abs{\nabla u_0 \cdot \nu}^2  \right ) = 0.
$$
Set
\begin{equation}\label{def_gamma_rho}
\gamma_r(\theta) = u_0(\rho(r),\theta).
\end{equation}
$\gamma_r$ represents a non-contractible curve in $\cal P$.  (\ref{poho_cons}) and $\lambda = \frac{L_0^2}{2\pi}$ give
$$
\int_0^{2\pi} \abs{\frac{d\gamma_r}{d\theta}}^2 = \frac{L_0^2}{2\pi} + \rho(r) \int_{\partial B_{\rho(r)}(a)}\abs{\nabla u_0 \cdot \nu}^2,
$$
from where we deduce that
$$
\lim_{r \to 0} \int_0^{2\pi} \abs{\frac{d\gamma_r}{d\theta}}^2 = \frac{L_0^2}{2\pi}.
$$
In particular, $\int_0^{2\pi} \abs{\frac{d\gamma_r}{d\theta}}^2$ is bounded in $r$.  Hence, we can choose a sequence $r_n \to 0$, as $n \to \infty$, and a curve $\gamma_0 \in W^{1,2}([0,2\pi];\cal P)$ such that $\gamma_n = \gamma_{r_n} \rightharpoonup \gamma_0$.  We have
$$
\frac{L_0^2}{2\pi}\leq \int_0^{2\pi} \abs{\frac{d\gamma_0}{d\theta}}^2 \leq \lim_{n \to \infty} \int_0^{2\pi} \abs{\frac{d\gamma_n}{d\theta}}^2 = \lim_{n \to \infty} \int_0^{2\pi} \abs{\frac{\partial u_0}{\partial \theta}(\rho(r_n), \theta)}^2\,d\theta  = \frac{L_0^2}{2\pi},
$$
which implies that $\gamma_0$ is a geodesic, and that $\gamma_n \to \gamma_0$ strong in $W^{1,2}([0,2\pi];\cal P).$
Let now
$$
\Lambda_n = \int_{\partial B_{\rho(r_n)}(a)} j(u_0)\cdot \tau.
$$
The convergences we just proved show that, as $n\to \infty$, $\Lambda_n \to \Lambda_0$, the anti-symmetric representative of $\gamma_0$.

\medskip
\medskip
\medskip

{\em Step 2---Analysis of $j(u_0)$.} Define now
$$
V(x) = j(u_0) - \frac{1}{2\pi r_a}\hat{\theta}_a\Lambda_0,
$$
where we interpret $\hat{\theta}_a\Lambda_0$ according to (\ref{vector_matrix_entries}).  The fact that $[u_0; \Delta u_0] = 0$ implies that
$$
{\rm div}(V) = 0.
$$
Observe that, for $R > r$, with the same choice of $\rho(r)$ we have made above, we know that
$$
\int_{\partial B_R(a)} V\cdot \nu = \int_{\partial B_{\rho(r)}(a)} V\cdot \nu =  \int_{\partial B_{\rho(r)}(a)} j(u_0) \cdot \nu \to 0 \,\,\,\, \mbox{as}\,\,\,\, r \to 0.
$$
Hence, there is a function $\psi_0 : \Omega \to M_a^3(\R{})$, such that $\psi_0 \in W^{1,2}(\Omega_{a,r}; M_a^3(\R{}))$ for all $r > 0$, and such that $V = \nabla ^\perp \psi_0$, that is,
$$
j(u_0) = \frac{1}{2\pi r_a}\hat{\theta}_a\Lambda_0 + \nabla^\perp \psi_0
$$
in almost all of $\Omega$.  As a matter of fact we have
$$
j(u_0) = \nabla^\perp \left ( -\frac{\ln(r_a)}{2\pi}\Lambda_0 + \psi_0 \right ).
$$
Denote
$$
\Psi_0 = -\frac{\ln(r_a)}{2\pi}\Lambda_0 + \psi_0 .
$$
Since $j(u_0) = \nabla^\perp \Psi_0$, we obtain
$$
-\Delta \Psi_0 = \nabla^\perp \cdot j(u_0) = -2\left [  \left [u_0; \frac{\partial u_0}{\partial x} \right ]; \left [u_0; \frac{\partial u_0}{\partial y}\right ]\right ] = -2\left [\frac{\partial \Psi_0}{\partial x}; \frac{\partial \Psi_0}{\partial y} \right ].
$$
In other words,
$$
\Delta \Psi_0 = 2\left [\frac{\partial \Psi_0}{\partial x}; \frac{\partial \Psi_0}{\partial y} \right ].
$$
The standard isomorphism of Lie Algebras between $\R{3}$ with the cross product and $M^3_a(\R{})$ with $[A;B]=AB-BA$, shows that this is the constant mean curvature (CMC) equation for $\Psi_0$ (see, for instance, \cite{riviere}).  Let us also observe that, by (\ref{upper_outer_ball}), for $r_1 > 0$, there is a constant $C > 0$, that depends only on $r_1$, such that
$$
\int_{\Omega \setminus B_{r_1}(a)} \abs{\nabla u_0}^2 = \int_{\Omega \setminus B_{r_1}(a)} \abs{j(u_0)}^2 = \int_{\Omega \setminus B_{r_1}(a)} \abs{\nabla \Psi_0}^2 \leq C.
$$
The proof of the regularity of the solutions of the CMC equation, as shown for instance in \cite{riviere}, shows then that $\Psi_0$ is smooth in $\Omega \setminus B_{r_1}(a)$.  However, it is easy to see that
$$
\int_\Omega \abs{\nabla \Psi_0}^2 = +\infty.
$$
In other words, we cannot apply the known results regarding regularity of solutions of the CMC equation to $\Psi_0$ in the whole of $\Omega$.

\medskip
\medskip
\medskip

{\em Step 3---Analysis of $\Psi_0$.} Now we show that
$$
\psi_0 = \Psi_0 + \frac{\ln(r_a)}{2\pi}\Lambda_0,
$$
satisfies $\psi_0 \in W^{1,2}(\Omega; M_a^3(\R{}))$, and also
$$
\Delta \psi_0 = 2\left [ \frac{\partial \psi_0}{\partial x}; \frac{\partial \psi_0}{\partial y}\right ] + \frac{1}{\pi r_a} \left [ \nabla \psi_0 \cdot \hat{\theta}_a ; \Lambda_0\right ].
$$
This, plus some extra work, will give us that $\psi_0 \in L^\infty(\Omega; M_a^3(\R{}))$.  To this end let us observe that
$$
j(u_0) = \frac{\hat{\theta}_a}{2\pi r_a}\Lambda_0 + \nabla^\perp \psi_0.
$$
In particular,
$$
\int_\Omega \abs{\nabla \psi_0 \cdot \hat{\theta}_a}^2 = \int_\Omega \abs{j(u_0)\cdot \hat{r}_a}^2 < +\infty.
$$
Next, a direct computation shows that
$$
\Delta \psi_0 = \nabla^\perp\cdot  j(\psi_0) + \frac{1}{\pi r_a} [\nabla \psi_0 \cdot \hat{\theta}_a; \Lambda_0].
$$
Dotting this equation into $\phi = \nabla \psi_0 \cdot B$, for $B(x) = x-a$, and following the standard Pohozaev's identity trick, we arrive at
\begin{align}
{\rm div}\left ( \left ( \frac{\abs{\nabla \psi_0}^2}{2} I_2 - D\psi_0^T D\psi_0 \right ) B \right ) &= -\frac{1}{\pi} \Lambda_0 \cdot [\nabla \psi_0\cdot \hat{r}_a; \nabla \psi_0 \cdot \hat{\theta}_a] \nonumber\\ 
&= -\frac{1}{2\pi}\Lambda_0 \cdot \left ( \nabla^\perp \cdot j(\psi) \right ) \label{poho_psi}.
\end{align}
Let us recall now from the end of Step 1, that we found a sequence $r_n \to 0$ as $n \to \infty$, with its corresponding $\rho(r_n)$ from (\ref{choice_rho}), such that 
$$
\gamma_n(\theta) = u_0(\rho(r_n), \theta) \to \gamma_0 \,\,\,\mbox{strongly in} \,\,\, W^{1,2}([0, 2\pi]; {\cal P}).
$$
We pick an arbitrary $r > 0$ with $\overline{B_r(a)}\subset \Omega$, any $n$ such that $r_n < r$, and integrate (\ref{poho_psi}) over $B_r(a) \setminus B_{\rho(r_n)}(a)$.  Since $\Lambda_0$ is constant, and writing $\rho_n$ instead of $\rho(r_n)$, we obtain
\begin{align}
r\int_{\partial B_{r}(a)} \left ( \frac{\abs{\nabla \psi_0}^2}{2}  - \abs{\nabla \psi_0 \cdot \nu}^2 \right ) &- \rho_n \int_{\partial B_{\rho_n}(a)} \left ( \frac{\abs{\nabla \psi_0}^2}{2}  - \abs{\nabla \psi_0 \cdot \nu}^2 \right ) \nonumber \\ 
&= -\frac{\Lambda_0}{2\pi} \cdot \int_{B_{r}(a) \setminus B_{\rho_n}(a)} \nabla^\perp \cdot j(\psi_0) \nonumber\\
&= -\frac{\Lambda_0}{2\pi} \cdot \int_{\partial (B_{r}(a) \setminus B_{\rho_n}(a))} j(\psi_0) \cdot \tau.\label{poho_for_psi_rho}
\end{align}
Let us recall now that
$$
j(u_0) = \frac{\hat{\theta}_a}{2\pi r_a} \Lambda_0 + \nabla^\perp \psi_0.
$$
This says that
\begin{align*}
\rho_n \int_{\partial B_{\rho_n}(a)} \abs{\nabla \psi_0\cdot \nu }^2 &= \rho_n \int_{\partial B_{\rho_n}(a)} \abs{\left ( j(u_0) - \frac{\hat{\theta}_a}{2\pi r_a} \Lambda_0\right ) \cdot \tau }^2 \\
&=  \int_0^{2\pi} \abs{\left [\gamma_n; \frac{d\gamma_n}{d\theta}\right ] - \frac{1}{2\pi} \Lambda_0 }^2 \to 0
\end{align*}
as $n \to \infty$.

Next, the definition of $j(u_0)$ shows that
$$
\rho_n\int_{\partial B_{\rho_n}(a)}\abs{\nabla \psi_0\cdot \tau}^2 = \rho_n\int_{\partial B_{\rho_n}(a)}\abs{\nabla u_0\cdot \nu}^2.
$$
By (\ref{choice_rho}),
$$
\rho_n\int_{\partial B_{\rho_n}(a)}\abs{\nabla \psi_0\cdot \tau}^2 \to 0
$$
 as $n \to \infty$.  

Now observe that, if $A\in M_a^3(\R{})$ is constant over $\partial B_{\rho_n}(a)$, then
$$
\int_{\partial B_{\rho_n}(a)} \left [ A; \nabla \psi_0 \cdot \tau \right ] = \left [ A; \int_{\partial B_{\rho_n}(a)} \nabla \psi_0 \cdot \tau \right ] = 0
$$
because $\partial B_{\rho_n}(a)$ is a closed curve.  We use this with $A = \psi_0^{\rho_n} = \psi_0(x_{\rho_n})$ for some fixed $x_{\rho_n} \in \partial B_{\rho_n}(a)$, to obtain
\begin{equation}\label{tangent_est}
\int_{\partial B_{\rho_n}(a)} j(\psi_0) \cdot \tau = \int_{\partial B_{\rho_n}(a)} \left [ \psi_0; \nabla \psi_0 \cdot \tau \right ] = \int_{\partial B_{\rho_n}(a)} \left [ \psi_0 - \psi_0^{\rho_n}; \nabla \psi_0 \cdot \tau \right ].
\end{equation}
Let now $x \in \partial B_{\rho_n}(a)$ and call $\Gamma(x_{\rho_n},x)$ a connected portion of $\partial B_{\rho_n}(a)$ that starts at $x_{\rho_n}$ and ends at $x$.  Obviously
$$
\psi_0(x) - \psi_0^{\rho_n} = \int_{\Gamma(x_{\rho_n},x)} \nabla\psi_0 \cdot \tau.
$$
Using this in (\ref{tangent_est}) we obtain
\begin{equation}\label{tangential_estimate_2}
\abs{\int_{\partial B_{\rho_n}(a)} j(\psi_0) \cdot \tau} \leq \left ( \int_{\partial B_{\rho_n}(a)} \abs{\nabla \psi_0 \cdot \tau} \right )^2 \leq 2\pi {\rho_n} \int_{\partial B_{\rho_n}(a)} \abs{\nabla \psi_0 \cdot \tau}^2.
\end{equation}
We conclude that
$$
\abs{\int_{\partial B_{\rho_n}(a)} j(\psi_0) \cdot \tau} \to 0
$$
as $n \to \infty$.  With all these facts we go back to (\ref{poho_for_psi_rho}) and let $n \to \infty$ to obtain
\begin{equation}\label{conseq_poho_0}
r\int_{\partial B_{r}(a)} \left  ( \frac{\abs{\nabla \psi_0}^2}{2}  - \abs{\nabla \psi_0 \cdot \nu}^2 \right ) = - \frac{\Lambda_0}{2\pi} \cdot \int_{\partial B_r(a)} j(\psi_0) \cdot \tau.
\end{equation}
This is valid for any $r > 0$ such that $\overline{B_r(a)} \subset \Omega$.

Observe now that the argument we used to obtain (\ref{tangential_estimate_2}) can be applied to any $r > 0$ with $\overline{B_r(a)}\subset \Omega$.  We use this on the right-hand side of (\ref{conseq_poho_0}) to obtain
$$
\int_{\partial B_r(a)} \abs{\nabla \psi_0\cdot \nu}^2 \leq C\int_{\partial B_r(a)}\abs{\nabla \psi_0 \cdot \tau}^2,
$$
where $C$ depends on $\Lambda_0$, but is independent of $r>0$.  We recall now that
$$
\int_{\Omega} \abs{\nabla u_0\cdot \nu}^2 < +\infty.
$$
This implies that
$$
\int_{B_r(a)}\abs{\nabla \psi_0 \cdot \tau}^2
$$
is finite, and hence
$$
\int_{B_r(a)}\abs{\nabla \psi_0}^2
$$
is also finite.

\medskip
\medskip
Once we know that $\psi_0 \in W^{1,2}(\Omega; M_a^3(\R{})),$ the last assertion of \thmref{main_theo} can be proved as follows. First notice that
$$
u_0\frac{\partial u_0}{\partial x}u_0 = u_0\frac{\partial u_0}{\partial y}u_0 = 0,
$$
because $u_0$ is $\cal P$-valued.  From here we obtain that
$$
\left [u_0; \frac{\partial u_0}{\partial x} \right ] = u_0 \left [u_0; \frac{\partial u_0}{\partial x} \right ] +  \left [u_0; \frac{\partial u_0}{\partial x} \right ]u_0,
$$
and the same holds for $ \frac{\partial u_0}{\partial y}$.  This last identity, the fact that
$$
j(u_0) = \frac{1}{2\pi r_a}(\Lambda_0 \hat{\theta}_a) + \nabla^\perp \psi_0,
$$
and some algebra show that
$$
Z_{u_0}(x) = \frac{1}{2\pi r_a}(\Lambda_0 - u_0\Lambda_0 - \Lambda_0 u_0)
$$
also satisfies
$$
Z_{u_0}(x) = -((\nabla^\perp \psi_0 \cdot \hat{\theta}_a) - u_0(\nabla^\perp \psi_0 \cdot \hat{\theta}_a) - (\nabla^\perp \psi_0 \cdot \hat{\theta}_a)u_0).
$$
Since $\psi_0 \in W^{1,2}(\Omega; M_a^3(\R{}))$, it follows that $Z_{u_0} \in L^2(\Omega; M_a^3(\R{}))$.
\medskip
\medskip
\medskip

{\em Step 4---Proof of the fact that $\psi_0 \in L^\infty(\Omega; M_a^3(\R{}))$.} From all we have done so far it is clear that $\psi_0$ is bounded, and in fact smooth, away from $a$.  All we need is to analyze $\psi_0$ near $a$.  To do this, recall that
$$
\int_{\Omega} \abs{\nabla \psi_0}^2
$$
is finite.  Hence, by adding a constant to $\psi_0$ we can also impose that
$$
\int_{\Omega} \abs{\psi_0}^2
$$
be finite.  In particular we have that
$$
\lim_{\delta \to 0} \int_{B_\delta(a)} (\abs{\nabla \psi_0}^2+\abs{\psi_0}^2) = 0.
$$
Now, for any $\delta > 0$ we can choose $\eta = \eta(\delta) \in [\delta/2, \delta]$ such that
\begin{align*}
\int_{\partial B_\eta(a)} (\abs{\nabla \psi_0}^2+\abs{\psi_0}^2) \leq \frac{2}{\delta}  \int_{(B_\delta\setminus B_{\delta/2})(a)} (\abs{\nabla \psi_0}^2+\abs{\psi_0}^2)
\end{align*}
Notice that
\begin{align*}
\left (\int_{\partial B_\eta(a)}(\abs{\nabla \psi_0} + \abs{\psi_0})\right )^2 \leq C \int_{(B_\delta\setminus B_{\delta/2})(a) } (\abs{\nabla \psi_0}^2+\abs{\psi_0}^2),
\end{align*}
and hence
$$
\lim_{\delta \to 0}  \int_{\partial B_\eta(a)}(\abs{\nabla \psi_0} + \abs{\psi_0}) = 0.
$$
Let us choose now $R > 0$ such that $B_{4R}(a) \subset \Omega$, and pick $b \in B_R(a)$, $b \neq a$.  Let also $\delta > 0$ be small, and set
$$
U_\delta = B_{2R}(b) \setminus ( B_\eta(a) \cup B_\eta(b)),
$$
where $\eta = \eta(\delta)$ is as explained before.  Call
$$
G(x,y) = \frac{1}{2\pi} \ln\left ( \frac{1}{\abs{x-y}}\right ),
$$
the fundamental solution of the Laplacian in the plane.  Observe now that, away from $a$, we have
\begin{align*}
\Delta \psi_0 &= 2\left [ \frac{\partial \psi_0}{\partial x}; \frac{\partial \psi_0}{\partial y}\right ] + \frac{1}{\pi r_a} \left [ \nabla \psi_0 \cdot \hat{\theta}_a ; \Lambda_0\right ] \\
&= \nabla^\perp \cdot  j(\psi_0) + {\rm div}\left ( \frac{\hat{\theta}_a}{\pi r_a} \left [\psi_0 ; \Lambda_0\right ] \right )
\end{align*}
Using Green's identity we obtain
\begin{align}
\int_{\partial U_\delta} (G(b,y) \nabla \psi_0 &- \psi_0 \nabla G(b,y))\cdot \nu \nonumber \\&= \int_{ U_\delta}G(b,y) \left ( \nabla^\perp \cdot j(\psi_0) + {\rm div}\left ( \frac{\hat{\theta}_a}{\pi r_a} \left [\psi_0 ; \Lambda_0\right ] \right )\right ) \nonumber  \\
&= - \int_{ U_\delta} \left ( \frac{\hat{\theta}_b}{2\pi r_b} \cdot j(\psi_0) - \frac{\hat{r}_b \cdot\hat{\theta}_a}{\pi^2 r_a r_b} \left [\psi_0 ; \Lambda_0\right ] \right )\nonumber  \\
& +\int_{ \partial U_\delta}G(b,y) \tau \cdot j(\psi_0) + \int_{ \partial U_\delta}\nu \cdot \frac{\hat{\theta}_a}{\pi r_a} \left [\psi_0 ; \Lambda_0\right ].\label{green_ident}
\end{align}
We intend to let $\delta \to 0$ in this last identity.  To do this, observe that
$$
 \int_{ \partial B_\eta(a) }\nu \cdot \frac{\hat{\theta}_a}{\pi r_a}\left [\psi_0 ; \Lambda_0\right ] = 0,
$$
and, since $\psi_0$ is smooth at $b$,
$$
\lim_{\delta \to 0} \int_{ \partial B_\eta(b) }\nu \cdot \frac{\hat{\theta}_a}{\pi r_a}\left [\psi_0 ; \Lambda_0\right ] = 0.
$$
Using again that $\psi_0$ is smooth away from $a$,we obtain
$$
\lim_{\delta \to 0} \int_{ \partial B_\eta(b) }G(b,y) \tau \cdot j(\psi_0) = 0.
$$
Now, if $\psi_0^\eta = \psi_0(x_\eta)$ for a fixed $x_\eta \in \partial B_\eta(a)$, and, for $x \in \partial B_\eta(a)$, $\Gamma(x_\eta, x)$ is the shortest portion of $\partial B_\eta(a)$ that starts in $x_\eta$ and ends in $x$, then
$$
\psi_0(x) - \psi_0^\eta = \int_{\Gamma(x_\eta,x)} \nabla \psi_0 \cdot \tau,
$$
so that
$$
\abs{\psi_0(x) - \psi_0^\eta} \leq \int_{\partial B_\eta(a)} \abs{\nabla \psi_0}.
$$
Writing
\begin{align*}
\int_{ \partial B_\eta(a)}G(b,y) \tau \cdot j(\psi_0) &= \int_{ \partial B_\eta(a)}(G(b,y)-G(b,a) \tau \cdot j(\psi_0) \\&+ G(b,a)\int_{ \partial B_\eta(a)} \tau \cdot j(\psi_0),
\end{align*}
we estimate
\begin{align*}
\abs{\int_{ \partial B_\eta(a)} \tau \cdot j(\psi_0)} &= \abs{\int_{ \partial B_\eta(a)} \left [ \psi_0 - \psi_0^\tau; \nabla \psi_0\cdot \tau\right ]} \\ &\leq \left ( \int_{ \partial B_\eta(a)} \abs{\nabla \psi_0}\right )^2 \leq 2\pi\eta \int_{ \partial B_\eta(a)} \abs{\nabla \psi_0}^2 \to 0,
\end{align*}
by the choice of $\eta$.  Again, this choice gives us that
$$
\abs{\int_{ \partial B_\eta(a)}(G(b,y)-G(b,a) \tau \cdot j(\psi_0)} \leq C(a,b)\delta \int_{\partial B_\eta(a)}\abs{\psi_0}\abs{\nabla \psi_0} \to 0
$$
as $\delta \to 0$ as well.

We notice next that, because $a \neq b$, $\frac{1}{r_ar_b} \in L^p(\Omega)$ for any $p < 2$.  Since $\psi_0 \in W^{1,2}(\Omega)$, we obtain that
$$
\frac{\hat{r}_b \cdot\hat{\theta}_a}{\pi^2 r_a r_b} \left [\psi_0 ; \Lambda_0\right ]  \in L^1(\Omega).
$$
Observe now that $G(b,y)$ is smooth near $a$, while $\psi_0$ is smooth near $b$.  Because of this, and the choice of $\eta$, we conclude that
$$
\lim_{\delta \to 0}\int_{\partial B_\eta(a)} G(b,y) \abs{\nabla \psi_0} = \lim_{\delta \to 0}\int_{\partial B_\eta(b)} G(b,y) \abs{\nabla \psi_0} = 0.
$$
Using again that $G(b,y)$ is smooth near $a$ and $\psi_0$ is smooth near $b$ we conclude that
$$
\lim_{\delta \to 0}\int_{\partial B_\eta(a)} \psi_0 \nabla G(b,y))\cdot \nu = 0,
$$
and
$$
\lim_{\delta \to 0}\int_{\partial B_\eta(b)} \psi_0 \nabla G(b,y))\cdot \nu = -\psi(b).
$$
Now we observe that, for $r > \eta > 0$ and $r$ small but fixed,
$$
\int_{(B_r \setminus B_\eta)(b)}\frac{\hat{\theta}_b}{2\pi r_b} \cdot j(\psi_0) = \int_\eta^r \frac{1}{2\pi r_b} \left ( \int_{\partial B_s(b)} j(\psi_0)\cdot \hat{\theta}_b \right ) ds.
$$
By the argument we gave to go from (\ref{tangent_est}) to (\ref{tangential_estimate_2}), this integral is bounded by $\int_\Omega \abs{\nabla \psi_0}^2$.  Putting all of this together in (\ref{green_ident}), as $\delta \to 0$ we obtain
\begin{align}
\psi_0(b) &= -\int_{\partial B_{2R}(b)} (G(b,y) \nabla \psi_0 - \psi_0 \nabla G(b,y))\cdot \nu \nonumber \\
&  -\int_{B_{2R}(b)} \left ( \frac{\hat{\theta}_b}{2\pi r_b} \cdot j(\psi_0) + \frac{\hat{r}_b \cdot\hat{\theta}_a}{\pi^2 r_a r_b} \left [\psi_0 ; \Lambda_0\right ] \right )\nonumber  \\
&+  \int_{ \partial B_{2R}(b)}G(b,y) \tau \cdot j(\psi_0) + \int_{ \partial B_{2R}(b)}\nu \cdot \frac{\hat{\theta}_a}{\pi r_a} \left [\psi_0 ; \Lambda_0\right ].\label{pre_bi_polar}
\end{align}
By the choice of $R$, because $b \in B_R(a)$ and because $\psi_0$ is smooth away from $a$, the boundary integrals above are uniformly bounded in $b$.  Next, we already mentioned that the term
$$
\int_{B_{2R}(b)} \frac{\hat{\theta}_b}{2\pi r_b} \cdot j(\psi_0)
$$
can be bounded by 
$$
\int_\Omega \abs{\nabla \psi_0}^2.
$$
Finally we need to estimate
$$
\int_{B_{2R}(b)} \frac{\hat{r}_b \cdot\hat{\theta}_a}{r_a r_b} \left [\psi_0 ; \Lambda_0\right ].
$$
For simplicity let us consider the case $b = (0,0)$, and $a=(\lambda,0)$.  Then
$$
\hat{r}_b = \frac{1}{r_b}(x,y)\,\,\,\mbox{and}\,\,\, \hat{\theta}_a = \frac{1}{r_a}(-y,x-\lambda),
$$
so that
$$
\int_{B_{2R}(b)} \frac{\hat{r}_b \cdot\hat{\theta}_a}{r_a r_b} \left [\psi_0 ; \Lambda_0\right ] = -\lambda \int_{B_{2R}(b)} \frac{y}{r_a^2 r_b^2} \left [\psi_0 ; \Lambda_0\right ].
$$
Write now $B_{2R}^+(b) = \{(x,y) \in B_{2R}(b) : y \geq 0\}$.  Then
\begin{align*}
\int_{B_{2R}(b)} \frac{\hat{r}_b \cdot\hat{\theta}_a}{r_a r_b} \left [\psi_0 ; \Lambda_0\right ] = &-\lambda \int_{B_{2R}^+(b)} \frac{y}{r_a^2 r_b^2} \left ( \left [\psi_0(x,y) ; \Lambda_0\right ]- \left [\psi_0(x,-y) ; \Lambda_0\right ] \right ) \\
= &-\lambda \int_{B_{2R}^+(b)} \frac{y}{r_a^2 r_b^2} \left ( \left [\psi_0(x,y) ; \Lambda_0\right ]- \left [\psi_0(x,0) ; \Lambda_0\right ] \right ) \\& -\lambda \int_{B_{2R}^+(b)} \frac{y}{r_a^2 r_b^2} \left ( \left [\psi_0(x,-y) ; \Lambda_0\right ]- \left [\psi_0(x,0) ; \Lambda_0\right ] \right ).
\end{align*}
We estimate the first of these two integrals, as the second obviously can be estimated in the same manner.  To do this observe that
\begin{align*}
&-\lambda \int_{B_{2R}^+(b)} \frac{y}{r_a^2 r_b^2} \left ( \left [\psi_0(x,y) ; \Lambda_0\right ]- \left [\psi_0(x,0) ; \Lambda_0\right ] \right ) \\ = &-\lambda \int_{-2R}^{2R} \int_{0}^{\sqrt{2R^2-x^2}} \frac{y}{r_a^2 r_b^2}  \left [\int_{0}^y\frac{\partial \psi_0}{\partial y}(x,s)ds ; \Lambda_0\right ]\,dy\,dx \\
= &-\lambda \int_{-2R}^{2R}  \int_{0}^{\sqrt{4R^2-x^2}} \left [\frac{\partial \psi_0}{\partial y}(x,s); \Lambda_0\right ]\int_{s}^{\sqrt{4R^2-x^2}}  \frac{y}{r_a^2 r_b^2}  dy\, ds\,dx
\end{align*}
Note next that
\begin{align*}
\int_{s}^{\sqrt{4R^2-x^2}}  &\frac{y}{r_a^2 r_b^2}  = \int_{s}^{\sqrt{4R^2-x^2}} \frac{y}{((x-\lambda)^2 + y^2)(x^2+y^2)} \\
&= \int_{s}^{\sqrt{4R^2-x^2}} \left (  \frac{1}{x^2 + y^2}  -  \frac{1}{(x-\lambda)^2+y^2}  \right ) \frac{y\,dy}{\lambda^2 - 2\lambda x} \\
&=\frac{1}{2(\lambda^2 -2\lambda x)}\left (  \ln\left(\frac{4R^2}{x^2+s^2}\right) -  \ln\left(\frac{\lambda^2-2\lambda x+4R^2}{(x-\lambda)^2+s^2}\right)  \right ) \\
&= \frac{1}{2(\lambda^2 -2\lambda x)}\left (  \ln\left(\frac{(x-\lambda)^2+s^2}{x^2+s^2}\right) -  \ln\left(\frac{\lambda^2-2\lambda x+4R^2}{4R^2}\right)  \right )
\end{align*}
Let us recall now the elementary estimate $\abs{\ln(1+t)}\leq C\abs{t}$, valid for $\abs{t}\leq 1/2$.  This says that, if $\abs{\lambda^2-2x \lambda} \leq 2R^2$, then
$$
\abs{\frac{1}{2(\lambda^2 -2\lambda x)} \ln\left(\frac{\lambda^2-2\lambda x+4R^2}{4R^2}\right)} \leq \frac{C}{R^2}
$$
On the other hand, if $\abs{\lambda^2-2x \lambda} \geq 2R^2$, then
$$
\abs{\frac{1}{2(\lambda^2 -2\lambda x)} \ln\left(\frac{\lambda^2-2\lambda x+4R^2}{4R^2}\right)} \leq \frac{1}{2R^2}\abs{ \ln\left(\frac{\lambda^2-2\lambda x+4R^2}{4R^2}\right)}.
$$
All this shows that
\begin{align*}
&\abs{\frac{1}{2(\lambda^2 -2\lambda x)} \int_{-2R}^{2R}  \int_{0}^{\sqrt{4R^2-x^2}} \left [\frac{\partial \psi_0}{\partial y}(x,s); \Lambda_0\right ] \ln\left(\frac{\lambda^2-2\lambda x+4R^2}{4R^2}\right)} \\
&\leq \frac{C}{R^2}\int_{B_{2R}} \abs{\frac{\partial \psi_0}{\partial y}(x,s)}  \max\left \{1; \abs{ \ln\left(\frac{\lambda^2-2\lambda x+4R^2}{4R^2}\right)}\right \},
\end{align*}
and the last integral is controlled by $\int_\Omega \abs{\nabla \psi_0}^2$.

We are left only with the integral
$$
I =  \int_{B_{2R}}\left [\frac{\partial \psi_0}{\partial y}(x,s); \Lambda_0\right ]  \frac{\lambda}{(\lambda^2 -2\lambda x)} \ln\left(\frac{(x-\lambda)^2+s^2}{x^2+s^2}\right)  \, ds\,dx.
$$
Now by Holder's inequality
$$
\abs{I} \leq C\left ( \int_{B_{2R}} \abs{\nabla \psi_0}^2 \int_{B_{2R}}\frac{\lambda^2}{(\lambda^2 -2\lambda x)^2} \ln^2\left(\frac{(x-\lambda)^2+s^2}{x^2+s^2}\right) \right )^{1/2}.
$$
We need then to estimate
\begin{align*}
\int_{B_{2R}}\frac{\lambda^2}{(\lambda^2 -2\lambda x)^2} \ln^2\left(\frac{(x-\lambda)^2+s^2}{x^2+s^2}\right) & \leq \int_{\R{2}}\frac{1}{(\lambda -2 x)^2} \ln^2\left(\frac{(x-\lambda)^2+s^2}{x^2+s^2}\right)  \\
&= \frac{1}{4} \int_{\R{2}}\frac{1}{x^2} \ln^2\left(\frac{(x-\frac{\lambda}{2})^2+s^2}{(x+\frac{\lambda}{2})^2+s^2}\right) 
\end{align*}
A simple way to see that this last integral is finite is to introduce bipolar coordinates, with poles at $(-\frac{\lambda}{2},0)$ and $(\frac{\lambda}{2},0)$.  These coordinates are
$$
\tau =\ln\left(\frac{(x-\frac{\lambda}{2})^2+s^2}{(x+\frac{\lambda}{2})^2+s^2}\right)
$$
and the angle $(-\frac{\lambda}{2},0)-(x,y)-(\frac{\lambda}{2},0)$, that we denote $\sigma$.  With these coordinates,
$$
x = \frac{\lambda \sinh(\tau)}{\cosh(\tau) - \cos(\sigma)}
$$
and
$$
dxds = \frac{\lambda^2}{(\cosh(\tau) - \cos(\sigma))^2}.
$$
Then
$$
 \int_{\R{2}}\frac{1}{x^2} \ln^2\left(\frac{(x-\frac{\lambda}{2})^2+s^2}{(x+\frac{\lambda}{2})^2+s^2}\right)  = \int_0^{2\pi}\int_{-\infty}^\infty \frac{\tau^2}{\sinh^2(\tau)}d\tau\,d\sigma = 2\pi \int_{-\infty}^\infty \frac{\tau^2}{\sinh^2(\tau)}d\tau,
 $$
 which is clearly finite.

\qed
\end{proof}

\bibliographystyle{plain}
\bibliography{tensor_nematic}

\begin{thebibliography}{10}

\bibitem{Ball_Zarn}
John~M. Ball and Arghir Zarnescu.
\newblock Orientability and energy minimization in liquid crystal models.
\newblock {\em Arch. Ration. Mech. Anal.}, 202(2):493--535, 2011.

\bibitem{bauman_phillips_park}
Patricia Bauman, Jinhae Park, and Daniel Phillips.
\newblock Analysis of nematic liquid crystals with disclination lines.
\newblock {\em Arch. Ration. Mech. Anal.}, 205(3):795--826, 2012.

\bibitem{BBH}
Fabrice Bethuel, Ha{\"{\i}}m Brezis, and Fr{\'e}d{\'e}ric H{\'e}lein.
\newblock {\em Ginzburg-{L}andau vortices}.
\newblock Progress in Nonlinear Differential Equations and their Applications,
  13. Birkh\"auser Boston Inc., Boston, MA, 1994.

\bibitem{canevari}
G.~{Canevari}.
\newblock {Biaxiality in the asymptotic analysis of a 2-D Landau-de Gennes
  model for liquid crystals}.
\newblock {\em ArXiv e-prints}, July 2013.

\bibitem{ericksen}
J.L. Ericksen.
\newblock Liquid crystals with variable degree of orientation.
\newblock {\em Archive for Rational Mechanics and Analysis}, 113(2):97--120,
  1991.

\bibitem{Helein_moving_frames}
P.~H\'elein, F.
\newblock {\em Harmonic maps, conservation laws and moving frames}, volume~7 of
  {\em Jindrich Necas Center for Mathematical modeling Lecture Notes}.
\newblock Cambridge University Press, cambridge, 2004.

\bibitem{jerrard_lower}
Robert~L. Jerrard.
\newblock Lower bounds for generalized {G}inzburg-{L}andau functionals.
\newblock {\em SIAM J. Math. Anal.}, 30(4):721--746, 1999.

\bibitem{CULect}
P.~Klapicky, editor.
\newblock {\em Topics in mathematical modelling and analysis}, volume~7 of {\em
  Jindrich Necas Center for Mathematical modeling Lecture Notes}.
\newblock MATFYZPRESS, Publishing House of the Faculty of Mathematics and
  Physics, Charles University in Prague, Prague, 2012.

\bibitem{apala_nonphys}
Apala Majumdar.
\newblock Equilibrium order parameters of nematic liquid crystals in the
  {L}andau-de {G}ennes theory.
\newblock {\em European J. Appl. Math.}, 21(2):181--203, 2010.

\bibitem{apala_zarnescu_01}
Apala Majumdar and Arghir Zarnescu.
\newblock Landau-{D}e {G}ennes theory of nematic liquid crystals: the
  {O}seen-{F}rank limit and beyond.
\newblock {\em Arch. Ration. Mech. Anal.}, 196(1):227--280, 2010.

\bibitem{Mottram_Newton}
N.J. Mottram and C.~Newton.
\newblock Introduction to ${Q}$-tensor theory.
\newblock Technical Report~10, Department of Mathematics, University of
  Strathclyde, 2004.

\bibitem{riviere}
T.~Rivi\`ere.
\newblock Integrability by compensation in the analysis of conformally
  invariant problems.
\newblock Minicourse PIMS, Vancouver, July 2009.

\bibitem{sandier}
Etienne Sandier.
\newblock Lower bounds for the energy of unit vector fields and applications.
\newblock {\em J. Funct. Anal.}, 152(2):379--403, 1998.

\bibitem{virga}
Epifanio~G. Virga.
\newblock {\em Variational theories for liquid crystals}, volume~8 of {\em
  Applied Mathematics and Mathematical Computation}.
\newblock Chapman \& Hall, London, 1994.

\end{thebibliography}

\section*{Appendix}

Here we provide an alternative argument to Lemma \ref{proj_lemma} demonstrating that the critical points of a slightly more general version of $E_\eps$ take values in the convex hull of $\cal P$ as long as their boundary data is in $\cal P$. 
\begin{lemma}
\label{max_princ_lemma}
For any critical point $u_\eps \in W^{1,2}(\Omega; F_1)$ of the functional 
$$E_\eps^\beta(u) = \int_\Omega \left ( \frac{\abs{\nabla u}^2}{2} + \frac{W_\beta(u)}{\eps^2}\right ),$$
such that $u|_{\partial\Omega}\in\cal P$, the function $u_\eps$ takes values in the convex hull of $\cal P$.
\end{lemma}
\begin{proof} We divide the proof into two steps.

{\em Step 1.} The potential $W_\beta$ defined in \eqref{wisborn} can be written as
$$
W_\beta(u)=\frac{1}{2}(1-\abs{u}^2)^2 - \frac{\beta}{6}(1-3\abs{u}^2+2{\rm tr}(u^3)),
$$
where ${|u|}^2=\tr{u^2}$. The gradient of the potential is
$$
(\nabla_u W_\beta)(u) = 2(\abs{u}^2-1)u + \beta(u-u^2),
$$
and its trace is
$$
{\rm tr}(\nabla_u W_\beta)(u)  = (\beta-2)(1-\abs{u}^2).
$$
We first compute the inner product
$$
\langle u; (\nabla_u W_\beta)(u) \rangle = 2(\abs{u}^2-1)\abs{u}^2 + \beta(\abs{u}^2 - {\rm tr}(u^3)), 
$$
and rewrite this expression as follows
$$
\langle u; (\nabla_u W_\beta)(u) \rangle = 2(\abs{u}^2-1)^2 + 2(\abs{u}^2-1) + \beta(\abs{u}^2 - {\rm tr}(u^3)).
$$
Then
\begin{align*}
\langle u; (\nabla_u W_\beta)(u) \rangle &= 2(\abs{u}^2-1)^2  - \frac{\beta}{2}(1-3\abs{u}^2+2{\rm tr}(u^3))\\
&+ (\abs{u}^2-1)\left ( 2 - \frac{\beta}{2}\right ).
\end{align*}
Comparing this with the definition of $W_\beta$ we obtain
$$
\langle u; (\nabla_u W_\beta)(u) \rangle = 4W_{\frac{3\beta}{4}} + (\abs{u}^2-1)\left ( 2 - \frac{\beta}{2}\right ).
$$
For critical points of 
$$E_\eps^\beta(u) = \int_\Omega \left ( \frac{\abs{\nabla u}^2}{2} + \frac{W_\beta(u)}{\eps^2}\right )$$
in $F_1$, we have
$$
-\Delta u + \frac{1}{\eps^2}(\nabla_u W_\beta)(u) = \lambda I_3,
$$
where the Lagrange multiplier 
$$
\lambda = \frac{(\beta-2)(1-\abs{u}^2)}{3}.
$$
Now set
$$
\alpha:=\frac{\abs{u}^2}{2}.
$$
The computations above show that
\begin{align*}
\Delta \alpha &= \abs{\nabla u}^2 + \langle u ; (\nabla_u W_\beta)(u)\rangle - \frac{(\beta-2)(1-\abs{u}^2)}{3} \\
&= \abs{\nabla u}^2 + \frac{1}{\eps^2}\left ( 4W_{\frac{3\beta}{4}} + (\abs{u}^2-1)\left (2+\frac{\beta -2}{3}-\frac{\beta}{2}\right ) \right ).
\end{align*}
In other words,
$$
\Delta \alpha =  \abs{\nabla u}^2 + \frac{1}{\eps^2}\left ( 4W_{\frac{3\beta}{4}} + \frac{8-\beta}{6}(\abs{u}^2-1) \right ).
$$
A standard maximum principle argument then shows that $\abs{u}\leq 1$ when $\beta \leq 8$.

\medskip
\medskip
\medskip

{\em Step 2.} We now follow the same line of reasoning for $\tau = \langle u; P\rangle$, where $P \in \cal P$ is a constant projection matrix.  We observe that
$$
\Delta \tau = \langle \Delta u; P\rangle = \frac{1}{\eps^2}\left ( 2(\abs{u}^2-1)\tau + \beta \tau -\beta \langle u^2;P\rangle - \frac{(\beta -2)(1-\abs{u}^2)}{3}  \right ).
$$
If we write $u$ in terms of its eigenvalues and eigenvectors as $u = \sum_{k=1}^n \lambda_kP_k$, then $\sum_{k=1}^3 \langle P;P_k\rangle = 1$ and $\langle P;P_k\rangle\geq 0$ for all $k=1, \ldots, 3$.  By Jensen's inequality
$$
\tau^2 = \left ( \sum_{k=1}^3 \lambda_k \langle P_k;P\rangle \right )^2 \leq \sum_{k=1}^3 \lambda_k^2\langle P_k;P\rangle = \langle u^2; P\rangle.
$$
Let us assume now $8\geq\beta > 2$.  Then $\abs{u}\leq 1$ from Step 1, and from the last inequality we obtain 
$$
\Delta \tau \leq \frac{\beta}{\eps^2}    \left ( 1-\frac{2}{\beta}(1-\abs{u}^2) - \tau  \right )\tau.
$$
Set $\xi = 1-\tau$.  Then $\xi$ has
$$
\Delta \xi \geq (\xi-1)\left ( \xi -\frac{2}{\beta}(1-\abs{u}^2)\right ).
$$
In particular, if $\Omega_> = \{x\in \Omega: \xi(x)>1\}$, the function $\xi$ is subharmonic in $\Omega_>$.  Since on $\partial \Omega$ we have $\xi = 1-\langle u;P\rangle \leq 1$, if $\Omega_>$ were nonempty (and also open), $\xi$ would attain a maximum in its interior.  By the maximum principle $\xi$ would be constant in $\Omega_>$---a contradiction.  We conclude that $\xi \leq 1$ in $\Omega$.

Since $\xi = 1-\langle u;P\rangle \leq 1$ in $\Omega$, we conclude that
$$
\langle u;P\rangle \geq 0
$$
for all $P \in \cal P$ fixed.  Therefore, when $\beta > 2$, the critical points $u_\eps$ of $E_\eps$ have $\lambda_3\geq 0$ and $u_\eps$ takes values in the convex hull of $\cal P$.
\qed
\end{proof}

\end{document}